\documentclass[a4paper,12pt,reqno,english]{amsart}
\usepackage[utf8]{inputenc}
\usepackage[T1]{fontenc}
\usepackage{babel}

\usepackage{aliascnt}
\usepackage[bookmarks=false,colorlinks,citecolor=blue,urlcolor=blue]{hyperref}
\usepackage[titletoc,toc,title]{appendix}
\usepackage{tikz}

\numberwithin{equation}{section}
\numberwithin{figure}{section}

\theoremstyle{plain}
  \newtheorem{theorem}{Theorem}

\theoremstyle{plain}
  \newaliascnt{proposition}{theorem}
  \newtheorem{proposition}[proposition]{Proposition}
  \aliascntresetthe{proposition}

\theoremstyle{plain}
  \newaliascnt{lemma}{theorem}
  \newtheorem{lemma}[lemma]{Lemma}
  \aliascntresetthe{lemma}

\theoremstyle{plain}
  \newaliascnt{corollary}{theorem}
  \newtheorem{corollary}[corollary]{Corollary}
  \aliascntresetthe{corollary}

\theoremstyle{definition}
  \newtheorem{definition}{Definition}

\theoremstyle{definition}
  \newtheorem{remark}{Remark}

\theoremstyle{definition}
  \newtheorem{notation}{Notation}

\date{}

\newcommand{\res}[1]{
\underset{#1}{\mathrm{Res}}
}

\newcommand{\sph}{\mu}

\newcommand{\modular}{\Delta_{h}}

\begin{document}

\title[An analogue of Weyl's law for quantized flag manifolds]{An analogue of Weyl's law for quantized irreducible generalized flag manifolds}

\author{Marco Matassa}%

\email{marco.matassa@gmail.com}

\address{Department of Mathematics, University of Oslo, P.B. 1053 Blindern, 0316 Oslo, Norway.}


\begin{abstract}
We prove an analogue of Weyl's law for quantized irreducible generalized flag manifolds.
By this we mean defining a zeta function, similarly to the classical setting, and showing that it satisfies the following two properties:
as a functional on the quantized algebra it is proportional to the Haar state; its first singularity coincides with the classical dimension. The relevant formulae are given for the more general case of compact quantum groups.
\end{abstract}

\maketitle

\section{Introduction}

The aim of this paper is to prove an analogue of Weyl's law for quantized irreducible generalized flag manifolds. We will be more precise about what we mean by this in a moment. The discussion of the problem will be of a general nature, with flag manifolds making an appearance only at the end of the paper.
To see what this should entail, we start by describing the classical results that we wish to generalize to the quantum setting.
In 1911 Weyl proved the following result for the asymptotic behaviour of the eigenvalues of the Laplace-Beltrami operator \cite{wey11}: let $M \subset \mathbb{R}^{d}$ be a bounded domain, here for simplicity without boundary, and let $N(\lambda)$ be the number of eigenvalues of the Laplace-Beltrami operator (counted with multiplicities) which are less or equal than $\lambda$; then we have the following equality
\begin{equation*}
\lim_{\lambda \to \infty} \frac{N(\lambda)}{\lambda^{d/2}} = \frac{V_{d}}{(2 \pi)^{d}} \textrm{vol}(M).
\end{equation*}
Here $V_{d}$ denotes the volume of the unit ball in $\mathbb{R}^{d}$.
It can be regarded as one of the first results in spectral geometry, since it allows to recover the dimension and the volume of $M$ from the knowledge of the spectrum of a certain operator.

This result can be reformulated and slightly generalized as follows.
Let $M$ be a closed Riemannian manifold of dimension $d$. Denote by $\Delta$ the Laplace-Beltrami operator defined with respect to a fixed metric. Then for any $f \in C^{\infty}(M)$ we have an equality, which in the following we will refer to as the \emph{residue formula}, given by
\begin{equation}\label{eq:res-form}
\res{z = d} \ \mathrm{Tr}(f \Delta^{-z/2}) = \frac{\Omega_{d}}{(2 \pi)^{d}} \int f \mathrm{dvol}.
\end{equation}
Here $\Omega_{d}$ denotes the volume of the $(d - 1)$-sphere. In this formulation the zeta function $\zeta_{\Delta, f}(z) =  \mathrm{Tr}(f \Delta^{-z/2})$ appears in place of the counting function $N(\lambda)$.

One way to prove the residue formula is by relating the residue at $z = d$ of the zeta function $\zeta_{\Delta, f}(z)$ to the Wodzicki residue. The latter can be easily computed using the principal symbol of $\Delta$.
When $f$ is equal to one we obtain a reformulation of Weyl's law. More generally, the residue formula shows that we can define integration in a purely spectral way.
Indeed this is one way of defining non-commutative integration for spectral triples \cite{con-book}.
The study of spectral triples gives one of the main motivation for this paper.

We make two remarks on Weyl's law, incarnated as the residue formula, from a modern perspective.
The first one is on the appearance of the dimension of $M$, which is not special to $\Delta$.
Indeed suppose that $P$ is an elliptic pseudo-differential operator of order $n > 0$, which extends to a positive self-adjoint operator.
Then it can be proven that $P^{-z}$ is trace-class for $\mathrm{Re}(z) > d/n$ and $\mathrm{Tr}(f P^{-z})$ is holomorphic on this open half-plane.
Therefore with $P^{-z/n}$ we would obtain a similar result.
The second remark is that, on the other hand, the proportionality of the residue with the integral of $f$ does not hold for any such operator.
Indeed it follows from the fact that the principal symbol of $\Delta$ contains the metric.

The residue formula holds in particular in the case of a compact Lie group $G$, with Lie algebra $\mathfrak{g}$.
Recall that the universal enveloping algebra $U(\mathfrak{g})$ can be identified with the algebra of left-invariant differential operators on $G$, with the center of $U(\mathfrak{g})$ corresponding to the algebra of bi-invariant differential operators.
Under this identification, the Laplace-Beltrami operator corresponds to the quadratic Casimir.
Therefore in this case the residue formula can be proven directly by making use of the representation theory of $\mathfrak{g}$.

This is the setting that we wish to consider.
Indeed, for universal enveloping algebras of semisimple Lie algebras, there exists a quantization procedure which endows them with a non-trivial Hopf algebra structure.
Similarly for algebras of representative functions on the corresponding Lie groups.
These objects, or more properly their completions as $C^{*}$-algebras, are referred to as quantum groups, for a reference see \cite{klsc}.
This quantization procedure can be extended also to homogeneous spaces, see \cite{stdi99} and references therein.

Our aim is then to investigate whether an analogue of the residue formula \eqref{eq:res-form} holds in the case of compact quantum groups and their homogeneous spaces.
There are two properties of this formula that we would like to mantain:
\begin{enumerate}
\item the proportionality between the (residue of the) trace of an operator and the integral,
\item the appearance of the dimension of the space as the first singularity of the zeta function.
\end{enumerate}
In the quantum setting some of the classical ingredients have to be replaced by their appropriate counterparts.
For example, the Haar integral of a compact Lie group has to be replaced by the Haar state, which satisfies analogues of the classical invariance conditions.
It is less clear how the Laplace-Beltrami operator should be replaced. For this reason we will only impose general properties at first, like being central and positive.
A general operator satisfying these properties will be denoted by $\mathcal{C}$. Later on we will be more specific about $\mathcal{C}$.

It turns out that also the trace must be replaced by the weight $\mathrm{Tr}(\cdot \modular)$, where $\modular$ is the modular operator of the Haar state. With this extra ingredient we define the zeta function $\zeta_{\mathcal{C}, a}(z) = \mathrm{Tr}(a \mathcal{C}^{-z/2} \modular)$, where $a$ is an element of the quantized algebra. Notice how this definition parallels the one appearing in the residue formula, except for the presence of $\modular$, which is trivial in the classical setting.
We simply write $\zeta_{\mathcal{C}}(z)$ in the case $a = 1$.
Then in this setting our first requirement, that of proportionality, can be recasted as an equality of certain zeta functions.
Indeed we will show that $\zeta_{\mathcal{C}, a}(z) = \zeta_{\mathcal{C}}(z) h(a)$, where $h$ is the Haar state, which expresses the sought after proportionality.
This equality will be valid for $\mathrm{Re}(z) > p$, where the number $p$ is called the \emph{spectral dimension}.
In the classical case this number coincides with the dimension of the manifold.
In the quantum case this number will depend non-trivially on the choice of the operator $\mathcal{C}$.
Therefore the main part of the paper will be devoted to the computation of the spectral dimension, for certain choices of the operator $\mathcal{C}$.

Before getting into that, let us briefly recall how this computation proceeds in the classical case.
In general there are two contributions: the eigenvalues of the Laplace-Beltrami operator and their multiplicities.
In the compact Lie group case, after making the identification with the quadratic Casimir, the eigenvalues and the multiplicities can be expressed in terms of the representation theory of the corresponding Lie algebra.
Indeed, upon using the Peter-Weyl decomposition, we see that the eigenvalues are the values of the Casimir in irreducible representations, while the multiplicities are the dimensions (squared) of these spaces.
These numbers can be computed by taking inner product of weights: the value of the Casimir in a irreducible representation of highest weight $\Lambda$ is given by $(\Lambda, \Lambda + 2 \rho)$, where $\rho$ is the half-sum of the positive roots; similarly, the dimensions of the representations can be computed via the Weyl dimension formula, which is expressed in terms of inner products of weights.

For compact quantum groups we still have an analogue of the Peter-Weyl decomposition. Then the values that $\mathcal{C}$ takes in irreducible representations will appear in the computation. Unfortunately, we can not say much about these values without knowing $\mathcal{C}$ more precisely.
We will also see that, due to the presence of the modular operator $\modular$, the dimensions of the irreducible representations are replaced by their quantum dimensions, a familiar notion in quantum group theory.
These values, on the other hand, can be computed by a quantum analogue of the Weyl dimension formula.
We will prove a simple asymptotic formula for the quantum dimension, given by $\dim_{q}(\Lambda) \sim q^{-(\Lambda, 2 \rho)}$. Here we write a dominant weight $\Lambda = \sum_{k = 1}^{r} n_{k} \omega_{k}$ in terms of the fundamental weights $\omega_{k}$, with the formula being valid for $n_{k}$ large.
The details of this asymptotic relation will be explained later.

To proceed we need to make some choice for the operator $\mathcal{C}$.
We will follow the construction of central elements for (quasi-triangular) Hopf algebras given in \cite{ligo92}. It is based on the universal R-matrix, which is one of the characteristic objects associated to a quantum group.
In this way we obtain a one-parameter family of central elements $\mathcal{C}_{t}$, which, in the classical limit, reduce to the quadratic Casimir up to rescaling.
The dependence on $t$ is introduced here to keep into account the non-linearity in the definition of the $q$-numbers.
This construction also depends on the choice of a fixed representation, which we will denote by $\Lambda_{0}$.

As a consequence of this construction, the value of $\mathcal{C}_{t}$ in some irreducible representation will be expressible in terms of inner products of weights. We denote the value in a representation of highest weight $\Lambda$ by $\chi_{\Lambda}(\mathcal{C}_{t})$. Then we will obtain the asymptotic formula $\chi_{\Lambda}(\mathcal{C}_{t}) \sim q^{- 4 t (\Lambda, \Lambda_{0})}$, where $\Lambda_{0}$ is the highest weight of the fixed representation.
Notice the non-linear dependence on the parameter $t$.
Using this asymptotic formula, together with the one for the quantum dimension, we will prove a general formula that computes the spectral dimension in terms of representation theory.
This result holds for the quantization of a compact simple Lie group. We will also show, in passing, a result regarding product spaces: in this case the spectral dimension is given by the sum of the spectral dimensions of the factors.

Now, in order to give explicit results, the only thing which is left to do is to compute some inner products of weights. This operation, though tedious, is completely classical. The result is somewhat disappointing: the spectral dimension does not seem to be related in any simple way to the classical one.
We should point out that, in the light of other results on various notions of dimension for compact quantum groups, this outcome is not too surprising.

On the other hand, quantized irreducible generalized flag manifolds are more well-behaved in this sense. Important results in this respect are given in the papers \cite{flagcalc1} and \cite{flagcalc2}: here it is shown that these spaces admit a canonical $q$-analogue of the de Rham complex, with the homogenous components having the same dimensions as in the classical case. Moreover, these differential calculi can be implemented in the sense of spectral triples by commutators with Dirac operators, see \cite{qflag} and \cite{qflag2}.
This is in contrast with the complexity of the differential calculi defined for compact quantum groups.

The results mentioned above provide the main motivation for restricting to the class of quantized irreducible generalized flag manifolds.
Indeed, as we will show, in this case our results take a particularly nice form.
To summarize these, we briefly recall some notions related to this class of spaces, which coincides with the class of compact irreducible Hermitian symmetric spaces. In this case the Peter-Weyl decomposition is multiplicity-free and the weights appearing in this decomposition are called dominant spherical weights. We will write such weights as $\Lambda_{S} = \sum_{k = 1}^{n} m_{k} \sph_{k}$, where $\sph_{k}$ are the fundamental spherical weights.

We start by giving the relevant results for the the inner products needed for the computation of the spectral dimension.
The first one is $(\Lambda_{S}, 2 \rho)$, which comes from the quantum dimension. We have the following interesting property: if we write $(\Lambda_{S}, 2\rho) = \sum_{k = 1}^{n} a_{k} m_{k}$ for some coefficients $a_{k}$, then we have $\frac{2}{(\theta, \theta)} \max \{ a_{k} \} = d$, where $\theta$ is the highest root and $d$ is the dimension of the space in consideration.
The second one is $(\Lambda_{S}, \Lambda_{0})$, which comes from the eigenvalues of $\mathcal{C}_{t}$. In this case we need to choose a representation $\Lambda_{0}$. We will consider both the fundamental and the adjoint representations. However the adjoint one turns out to be more appropriate for our purposes. In this case the highest weight is given by the highest root $\theta$ and we have the simple result $(\Lambda_{S}, \theta) = (\theta, \theta) \sum_{k = 1}^{n} m_{k}$.

Using these results, we can compute the spectral dimension $p$. The outcome of the computation is that $p = \frac{d}{4 t}$, that is the spectral dimension is proportional to the classical one.
This holds for all quantized irreducible generalized flag manifolds and for any value of $t$. In particular we are free to choose $t = 1/4$ in such a way that the two dimensions coincide. In this way we obtain the promised analogue of Weyl's law, which is formulated in \autoref{theorem:weyl-law}. Other values of $t$ might also be of interest, but we will not explore this here.

The structure of the paper is as follows. In \autoref{section:quantum-groups} we summarize basic results on compact quantum groups and their homogeneous spaces. In \autoref{section:integration} we describe our setting, define the relevant zeta functions and show the proportionality with the Haar state. We also give an asymptotic formula for the quantum dimension.
In \autoref{section:central} we discuss the choice of a central element. We introduce a one-parameter family of such elements and give an asymptotic formula for their eigenvalues.
In \autoref{section:spec-dim} we prove some general results on the spectral dimension.
In \autoref{section:comp-groups} we perform some computations of inner products for compact quantum groups.
In \autoref{section:flag-man} we recall some results on irreducible generalized flag manifolds and compute the corresponding quantum dimension.
Finally in \autoref{section:comp-flag} we finish the computations for this class of spaces, by analyzing the case of the fundamental and adjoint representations.

\section{Quantum groups and homogeneous spaces}\label{section:quantum-groups}

Let $\mathfrak{g}$ be a finite-dimensional complex semisimple Lie algebra.
Let $(\cdot, \cdot)$ be a multiple of the Killing form and fix a system $\{ \alpha_{1}, \cdots, \alpha_{r} \}$ of simple roots.
Denote by $(a_{ij})$ the Cartan matrix and set $d_{i} = (\alpha_{i}, \alpha_{i})/2$.
Using this data we can define the quantized universal enveloping algebra $U_q(\mathfrak{g})$, for a reference see the book \cite{klsc}.

Fix $0 < q < 1$. The quantized universal
enveloping algebra $U_q(\mathfrak{g})$ is generated by the elements $E_i$, $F_i$, $K_i$,
$K_i^{-1}$, $1\le i\le r$, satisfying the relations
$$
K_iK_i^{-1}=K_i^{-1}K_i=1,\ \ K_iK_j=K_jK_i,
$$
$$
K_iE_jK_i^{-1}=q_i^{a_{ij}}E_j,\ \
K_iF_jK_i^{-1}=q_i^{-a_{ij}}F_j,
$$
$$
E_iF_j-F_jE_i=\delta_{ij}\frac{K_i-K_i^{-1}}{q_i-q_i^{-1}},
$$
plus the quantum analogue of the Serre relations
$$
\sum^{1-a_{ij}}_{k=0}(-1)^k\begin{bmatrix}1-a_{ij}\\
k\end{bmatrix}_{q_i} E^k_iE_jE^{1-a_{ij}-k}_i=0,\ \
\sum^{1-a_{ij}}_{k=0}(-1)^k\begin{bmatrix}1-a_{ij}\\
k\end{bmatrix}_{q_i} F^k_iF_jF^{1-a_{ij}-k}_i=0,
$$
where $\displaystyle\begin{bmatrix}m\\
k\end{bmatrix}_{q_i}=\frac{[m]_{q_i}!}{[k]_{q_i}![m-k]_{q_i}!}$,
$[m]_{q_i}!=[m]_{q_i}[m-1]_{q_i}\dots [1]_{q_i}$,
$\displaystyle[n]_{q_i}=\frac{q_i^n-q_i^{-n}}{q_i-q_i^{-1}}$ and
$q_i=q^{d_i}$. This is a Hopf $*$-algebra with coproduct $\Delta$ and
counit $\varepsilon$ defined by
$$
\Delta(K_i)=K_i\otimes K_i,\ \
\Delta(E_i)=E_i\otimes1+ K_i\otimes E_i,\ \
\Delta(F_i)=F_i\otimes K_i^{-1}+1\otimes F_i,
$$
$$
\varepsilon(E_i)=\varepsilon(F_i)=0,\ \ \varepsilon(K_i)=1,
$$
and with involution given by $K_i^*=K_i$, $E_i^*=F_iK_i$, $F_i^*=K_i^{-1}E_i$.
There are several other different presentations of the quantized enveloping algebra $U_q(\mathfrak{g})$, but our results will be essentially independent of such a presentation.

Let $P$ be the weight lattice and $P^{+}$ the cone of dominant weights. We denote by $\{ \omega_{1}, \cdots, \omega_{r} \}$ the fundamental weights with respect to the fixed choice of simple roots $\{ \alpha_{1}, \cdots, \alpha_{r} \}$.
For a left $U_{q}(\mathfrak{g})$-module $V$, we say that $v \in V$ has weight $\mu$ if $K_{i} v = q^{(\mu, \alpha_{i})} v$ and write $V_{\mu}$ for the corresponding weight space.
The (admissible) irreducible finite-dimensional representations of $U_{q}(\mathfrak{g})$ can be labelled by their highest weights $\Lambda \in P^{+}$. The corresponding modules are denoted by $V(\Lambda)$ and have the weight decomposition
\begin{equation*}
V(\Lambda) = \bigoplus_{\mu \leq \Lambda} V(\Lambda)_{\mu}, \qquad
V(\Lambda)_{\mu} = \{ v \in V(\Lambda) : K_{i} v = q^{(\mu, \alpha_{i})} v \},
\end{equation*}
where $\leq$ is the dominance order on $P$ with respect to the positive roots.

The quantized function algebra $\mathbb{C}_{q}[G]$ can be defined as the Hopf $*$-subalgebra of the dual of $U_{q}(\mathfrak{g})$ spanned by the matrix coefficients of the finite dimensional $U_{q}(\mathfrak{g})$-representations.
Then the analogue of the Peter-Weyl theorem is given by the decomposition
\begin{equation}\label{eq:pet-wey}
\mathbb{C}_{q}[G] = \bigoplus_{\Lambda \in P^{+}} V(\Lambda) \otimes V(\Lambda)^{*}.
\end{equation}
The algebra $\mathbb{C}_{q}[G]$ is a $U_{q}(\mathfrak{g})$-bimodule with respect to the left and right actions given in terms of the pairing. In the following we will denote them by $\triangleright$ and $\triangleleft$.

Homogeneous spaces can also be quantized in a similar way, see \cite{stdi99} and references therein. The quantized function algebra $\mathbb{C}_{q}[G/K]$ is given by those elements of $\mathbb{C}_{q}[G]$ which are infinitesimally invariant under the action of $U_{q}(\mathfrak{k})$, that is
\begin{equation*}
\mathbb{C}_{q}[G/K] = \{ a \in \mathbb{C}_{q}[G] : X \triangleright a = \varepsilon(X) a, \ \forall X \in U_{q}(\mathfrak{k}) \}.
\end{equation*}
Since the left and right actions commute, the algebra $\mathbb{C}_{q}[G/K]$ is a right $U_{q}(\mathfrak{g})$-module.
In particular, by restricting to those representations that admit $K$-invariant vectors, we can obtain a decomposition similar to \eqref{eq:pet-wey}. We will come back to this point later on.

\section{Integration and quantum dimension}
\label{section:integration}

\subsection{Setting and notation}

In this section we consider a general quantized function algebra $\mathbb{C}_{q}[G]$, where $G$ is a compact simple Lie group. We investigate a quantum analogue of the first feature of the residue formula mentioned in the introduction. This is the proportionality between the integral and the (residue of) the trace of some operator.
In the quantum setting the integral must be replaced by the Haar state, which satisfies invariance properties which are the analogue of those of the Haar integral in the classical setting.

We now recall some of its properties. It is the unique state $h$ on $\mathbb{C}_{q}[G]$ which satisfies the properties $(\mathrm{id} \otimes h) \Delta(a) = (h \otimes \mathrm{id}) \Delta(a) = h(a) 1$ for all $a \in \mathbb{C}_{q}[G]$, which can be seen to be equivalent to the invariance properties in the classical case.
It does not satisfy the trace property. The action of its modular group can be expressed in terms of the action of $U_{q}(\mathfrak{g})$.
It takes the form $h(a b) = h(b \theta(a))$, where $\theta(a) = K_{2\rho} \triangleright a \triangleleft K_{2\rho}$ is written in terms of the left and right actions of $U_{q}(\mathfrak{g})$ on $\mathbb{C}_{q}[G]$.
Here $K_{2\rho}$ is the element which implements the square of the antipode. The notation $\rho$ refers to the half-sum of the positive roots.
Using the GNS construction for the Haar state, the algebra $\mathbb{C}_{q}[G]$ can be completed to a Hilbert space.

While there is a clear replacement for the integral, this is not the case for the other ingredient of the residue formula, the Laplace-Beltrami operator.
Indeed in the quantum setting we do not have a good notion of pseudo-differential calculus, so that it is not clear how to proceed.
We will return to the problem of making such a choice in a later section.
For the moment we will just require some general properties, in analogy with the classical case.
Recall that the Laplace-Beltrami operator can be identified with a central (quadratic) element in $U(\mathfrak{g})$, and is a positive operator on the Hilbert space $L^{2}(G)$.
As a replacement in the quantum setting we take an element $\mathcal{C} \in U_{q}(\mathfrak{g})$ (or possibly in some completion). We assume that:
\begin{itemize}
\item $\mathcal{C}$ is a central element in $U_{q}(\mathfrak{g})$,
\item $\mathcal{C}^{*} = \mathcal{C}$ with respect to the involution in $U_{q}(\mathfrak{g})$.
\end{itemize}

We can consider $\mathcal{C}$ to be acting on $\mathbb{C}_{q}[G]$ via the right action $\triangleleft$. Here we use the right action since for homogeneous spaces we will require invariance under the left action. In any case for central elements like $\mathcal{C}$ the two actions coincide.
It becomes an unbounded operator with dense domain $\mathbb{C}_{q}[G]$ on the corresponding Hilbert space $L^{2}_{q}(G)$.
It is easy to see that, due to the structure of $\mathbb{C}_{q}[G]$ and our reality assumption, it extends to a self-adjoint operator which we denote by the same symbol $\mathcal{C}$. Finally we ask that:
\begin{itemize}
\item $\mathcal{C}$ is a positive operator on the Hilbert space $L^{2}_{q}(G)$.
\end{itemize}

We also assume for simplicity that the operator $\mathcal{C}$ is invertible, but this assumption can be easily removed by standard arguments.

As discussed in \cite{mat14}, the lack of the trace property for the Haar state must be accounted for in the residue formula.
This means introducing explicitely the modular operator implementing the automorphism $\theta$, which we denote by $\modular$.
This will be taken into account by our definition of the relevant zeta functions.

\begin{definition}
We define the \emph{zeta function associated to} $\mathcal{C}$ by $\zeta_{\mathcal{C}}(z) = \mathrm{Tr}(\mathcal{C}^{-z/2} \modular)$.
Similarly, for $a \in \mathbb{C}_{q}[G]$ we define $\zeta_{\mathcal{C}, a}(z) = \mathrm{Tr}(a \mathcal{C}^{-z/2} \modular)$.
Here as usual $a$ denotes the operator of left multiplication by $a$, that is we omit the representation symbol.
\end{definition}

As they stand, these zeta functions are not necessarily well-defined. We need some summability condition for the relevant operators. This leads to the following definition.

\begin{definition}
We define the \emph{spectral dimension} associated to the zeta function $\zeta_{\mathcal{C}}(z)$ to be the number $p = \inf \{ s \geq 0 : \zeta_{\mathcal{C}}(s) < \infty \}$, when it exists.
\end{definition}

Therefore the zeta function $\zeta_{\mathcal{C}}(z)$ will be defined at least for $\{ s \in \mathbb{R} : s > p \}$.
It actually follows immediately from Hölder's inequality that this can be extended to $\{ z \in \mathbb{C} : \mathrm{Re}(z) > p \}$.
These statements hold true also for the zeta function $\zeta_{\mathcal{C}, a}(z)$.

\subsection{Proportionality with the Haar state}

In the setting that we have described above, we will find that there is a simple relation between the two zeta functions $\zeta_{\mathcal{C}}$ and $\zeta_{\mathcal{C}, a}$.
Before getting into that, we note a simple result on the holomorphicity of such functions, which holds in the half-plane defined by the spectral dimension.

\begin{lemma}
Let $p$ be the spectral dimension associated to the zeta function $\zeta_{\mathcal{C}}$. Then $\zeta_{\mathcal{C}}(z)$ is holomorphic for all $z \in \mathbb{C}$ such that $\mathrm{Re}(z) > p$.
The same holds $\zeta_{\mathcal{C}, a}$.
\end{lemma}
\begin{proof}
The statement will follow from Weierstrass' theorem (or equivalently Morera's theorem).
We briefly recall its formulation. Suppose we have a series $u(z) = \sum_{k = 0}^{\infty} u_{k}(z)$ which converges uniformly on compact sets inside a domain $D$. Moreover suppose that $u_{k}$ are holomorphic functions in $D$. Then the function $u(z)$ is holomorphic in $D$.

In our case the domain will be the half-plane $\{ z \in \mathbb{C} : \mathrm{Re}(z) > p \}$. It is clear, by writing the trace explicitely, that the terms of our series are holomorphic in this domain.
Then we have to show that it converges uniformly on compact sets. To do this, let us fix some $\epsilon > 0$ and write $z = s + i u$ with $s \geq p + \epsilon$.
Then we have the simple inequality
\begin{equation*}
| \mathrm{Tr}(\mathcal{C}^{-z/2} \modular) |
= | \mathrm{Tr}(\mathcal{C}^{-(s - (p + \epsilon) + i u)/2} \mathcal{C}^{-(p + \epsilon)/2} \modular) |
\leq \| \mathcal{C}^{-(s - (p + \epsilon))/2} \| \mathrm{Tr}(\mathcal{C}^{-(p + \epsilon)/2} \modular).
\end{equation*}
By taking the maximum of $\| \mathcal{C}^{-(s - (p + \epsilon))/2} \|$ over the chosen compact set we obtain the uniform bound.
Therefore we conclude that $\zeta_{\mathcal{C}}(z)$ is holomorphic for $\mathrm{Re}(z) > p$.
To prove the analogue result for $\zeta_{\mathcal{C}, a}(z)$ we only need to observe that $a$ acts as a bounded operator.
\end{proof}

In the next proposition we state the relation between the zeta functions $\zeta_{\mathcal{C}}(z)$ and $\zeta_{\mathcal{C}, a}(z)$.
This follows from a result which, in some form or another, is well known to the experts.
Here we provide a minor generalization of this result as an equality of holomorphic functions.

\begin{proposition}
\label{prop:res-int}
Let $\mathcal{C}$ be an operator satisfying the properties listed previously. Let $p$ be the spectral dimension of the corresponding zeta function.
Then we have $\zeta_{\mathcal{C}, a}(z) = \zeta_{\mathcal{C}}(z) h(a)$ for any $a \in \mathbb{C}_{q}[G]$ and for any $z \in \mathbb{C}$ such that $\mathrm{Re}(z) > p$.
Here $h$ is the Haar state.
\end{proposition}

\begin{proof}
We start by proving the statement for $z = s \in \mathbb{R}$.
By assumption $\mathcal{C}^{-s/2}$ is a positive operator and commutes with all of $U_{q}(\mathfrak{g})$.
Therefore the result follows from \cite[Lemma 2.1]{netu05}, although stated in a slightly different language.
Since this lemma is given without proof in the paper, we mention that the proof can be obtained by using the results of \cite[Section 1.4]{netu04} together with the uniqueness of the Haar state $h$.

The complex version of the statement easily follows from the previous lemma. Indeed the two zeta functions are holomorphic in the open half-plane $\{ z \in \mathbb{C} : \mathrm{Re}(z) > p \}$ and satisfy $\zeta_{\mathcal{C},a}(z) = \zeta_{\mathcal{C}}(z) h(a)$ for $z$ in the open half-line $(0,+\infty)$. But then they must coincide in the entire half-plane, by the uniqueness of analytic continuation.
\end{proof}

The above result can be regarded as the quantum counterpart of the classical result for these zeta functions.
Namely we find the proportionality of the trace of some operator with the Haar state, which provides the natural notion of integration in this context.
We should stress that this would not be the case had we omitted the modular operator $\modular$.
On the other hand this result is essentially independent of the choice of $\mathcal{C}$.
What changes, however, is the location of the singularity of $\zeta_{\mathcal{C}}(z)$, that is the spectral dimension.
This will be investigated in detail in later sections, where making a specific choice for $\mathcal{C}$ will be important.

\subsection{Quantum dimension}

We can gain some additional insight into the zeta function $\zeta_{\mathcal{C}}$ by using the analogue of the Peter-Weyl theorem for compact quantum groups. We can write $\mathbb{C}_{q}[G]$ as a sum over irreducible representations, which are the same as in the classical case.
Since the element $\mathcal{C}$ is central, it takes the constant value $\chi_{\Lambda}(\mathcal{C})$ in any irreducible representation of highest weight $\Lambda$.
Therefore the trace formula can be rewritten as
\begin{equation*}
\zeta_{\mathcal{C}}(z)
= \mathrm{Tr}(\mathcal{C}^{-z/2} \modular)
= \sum_{\Lambda \in P^{+}} \mathrm{Tr}_{V(\Lambda) \otimes V(\Lambda)^{*}}(\modular) \chi_{\Lambda}(\mathcal{C})^{-z/2}.
\end{equation*}
We have denoted explicitely the fact that the trace of $\modular$ should be taken over the vector space $V(\Lambda) \otimes V(\Lambda)^{*}$. In the classical case, with the modular operator reducing to the identity, this trace would give the dimension of this vector space, that is $(\dim V(\Lambda))^{2}$. In the quantum case the value of $\mathrm{Tr}_{V(\Lambda)}(\modular) = \dim_{q} V(\Lambda)$ is known as  the \emph{quantum dimension} of $V(\Lambda)$ (also called intrinsic dimension in more categorial approaches). It can be proven that it satisfies the property $\dim_{q}(V \otimes W) = \dim_{q} V \cdot \dim_{q} W$ for any two modules $V, W$.

It is possible to obtain an explicit formula for the quantum dimension in terms of the roots of $\mathfrak{g}$, see \cite[Lemma 1]{zgb91}.
Denote by $\rho$ the half-sum of the positive roots, or alternatively the sum of the fundamental weights. Then we have
\begin{equation*}
\dim_{q}(\Lambda) = \prod_{\alpha > 0} \frac{q^{(\Lambda + \rho, \alpha)} - q^{-(\Lambda + \rho, \alpha)}}{q^{(\rho, \alpha)} - q^{-(\rho, \alpha)}}.
\end{equation*}
Here the product is taken over all positive roots. Notice that for $q \to 1$ it reduces to the usual Weyl dimension formula. We now wish to obtain an asymptotic formula for the quantum dimension for large $\Lambda$. What this means is explained below.

\begin{notation}
\label{nota:asymp}
We will write $f(n) \sim g(n)$ to mean that $C_{1} |g(n)| \leq |f(n)| \leq C_{2} |g(n)|$ for large enough natural numbers $n = \{ n_{1}, \cdots, n_{r} \}$, where $C_{1}$ and $C_{2}$ are two positive constants. In other words, $f(n) \sim g(n)$ if $f(n) = O(g(n))$ and $g(n) = O(f(n))$ for $n \to \infty$.
\end{notation}

We will also write $\Lambda = \sum_{k = 1}^{r} n_{k} \omega_{k}$ for a general dominant weight, where $\{ \omega_{k} \}$ are the fundamental weights.
We can now easily give the asymptotic formula.

\begin{lemma}\label{lem:qdim}
Let $V(\Lambda)$ be an $U_{q}(\mathfrak{g})$-module. Then $\dim_{q}(\Lambda) \sim q^{-(\Lambda, 2\rho)}$.
\end{lemma}
\begin{proof}
Since the scalar product $(\Lambda, \alpha)$ is positive and $0 < q < 1$, it follows that for $n \to \infty$ the term $q^{(\Lambda, \alpha)}$ goes to zero while the term $q^{-(\Lambda, \alpha)}$ grows large.
Therefore we have
\begin{equation*}
\frac{q^{(\Lambda + \rho, \alpha)} - q^{-(\Lambda + \rho, \alpha)}}{q^{(\rho, \alpha)} - q^{-(\rho, \alpha)}} \sim q^{-(\Lambda, \alpha)}.
\end{equation*}
The same holds when taking products, so that
\begin{equation*}
\dim_{q}(\Lambda)
\sim \prod_{\alpha > 0} q^{-(\Lambda, \alpha)}
= q^{- \sum_{\alpha > 0} (\Lambda, \alpha)}.
\end{equation*}
Using the formula $\rho = \frac{1}{2} \sum_{\alpha > 0} \alpha$ the conclusion follows.
\end{proof}

\section{Choice of a central element}
\label{section:central}

\subsection{Discussion of the choices}

In this section we discuss the choice of a central element that should play the role of the quadratic Casimir in the quantum setting.
Let us recall that in the classical case the Casimir element is the central element that takes the value $(\Lambda, \Lambda + 2 \rho)$ in an irreducible representation of highest weight $\Lambda$.
It is the unique central element which is quadratic in the generators, up to rescaling.
Unfortunately in the quantum case there is no canonical choice of such a central element, at least for our purposes.

For this reason we will consider a general procedure to obtain central elements, which is given in \cite{ligo92}.
Suppose we have an element $W \in U_{q} (\mathfrak{g}) \otimes U_{q} (\mathfrak{g})$ (strictly speaking in some completion) such that $\Delta(x) W = W \Delta(x)$ for all $x \in U_{q} (\mathfrak{g})$.
Let us fix a representation $V(\Lambda_{0})$ of highest weight $\Lambda_{0}$.
Define the $q$-trace $\tau_{q}: U_{q} (\mathfrak{g}) \otimes U_{q} (\mathfrak{g}) \to U_{q} (\mathfrak{g})$ by
\begin{equation*}
\tau_{q}(x \otimes y) = x \mathrm{Tr}_{V(\Lambda_{0})} ( \pi_{\Lambda_{0}} (K_{2\rho} y) ),
\end{equation*}
where $\pi_{\Lambda_{0}}$ and $\mathrm{Tr}_{V(\Lambda_{0})}$ denote the representation and the trace on the vector space $V(\Lambda_{0})$.
Then it can be proven that $\tau_{q}(W)$ is a central element in $U_{q} (\mathfrak{g})$.

The value of the central element $\tau_{q}(W)$ acting on a representation $V(\Lambda)$ of highest weight $\Lambda$ can be determined explicitely.
We denote this value by $\chi_{\Lambda}(\tau_{q}(W))$.
First we introduce some notation.
Let $V(\mu)$ be an irreducible representation with highest weight $\mu$ which occurs in the decomposition of $V(\Lambda) \otimes V(\Lambda_{0})$.
We can write $\mu = \Lambda + \lambda$, where $\lambda$ is a weight of $V(\Lambda_{0})$.
Suppose that $W$ acts as the scalar $w_{\lambda}$ on $V(\Lambda + \lambda)$.
Then it can be proven that
\begin{equation}\label{eq:value-c}
\chi_{\Lambda}(\tau_{q}(W))
= \sum_{\lambda} w_{\lambda} \frac{\dim_{q} V(\Lambda + \lambda)}{\dim_{q} V(\Lambda)},
\end{equation}
where the sum is over all the weights $\lambda$ of $V(\Lambda_{0})$ (with their multiplicities).

In particular, since $U_{q} (\mathfrak{g})$ is a quasi-triangular Hopf algebra, it has a universal $R$-matrix $\mathcal{R}$ which can be used for this construction.
Indeed, from the defining properties of $\mathcal{R}$, it follows that $\Delta(x) \mathcal{R}^{T} \mathcal{R} = \mathcal{R}^{T} \mathcal{R} \Delta(x)$, where $\mathcal{R}^{T}$ denotes the transpose of $\mathcal{R}$.
In \cite{ligo92} it is shown that for $W = (\mathcal{R}^{T} \mathcal{R} - 1) / (q - q^{-1})$ the invariants $\tau_{q}(W^{m})$, for $m \in \mathbb{N}$, reduce in the classical limit $q \to 1$ to the Gel'fand invariants for the Lie algebra $\mathfrak{g}$.

In the following we will use this general construction to define a central element $\mathcal{C}$, which will satisfy the requests of the previous section. We will however make a slightly more general choice for the element $W$. In particular we will define a one-parameter family of central elements $\mathcal{C}_{t}$. Before doing that, we must recall some properties of the $R$-matrix, which will motivate our generalization.
Specifically we are interested in the combination $\mathcal{R}^{T} \mathcal{R}$, since it commutes with coproducts.
First of all we have \cite[Section 8.4.3, Proposition 21]{klsc}
\begin{equation*}
\mathcal{R}^{T} \mathcal{R} = \Delta(q^{C}) (q^{-C} \otimes q^{-C}),
\end{equation*}
where $C$ is the central element which acts by multiplication by $(\Lambda, \Lambda + 2 \rho)$ in a representation of highest weight $\Lambda$. From this expression it is easy to see the semiclassical limit of $\mathcal{R}^{T} \mathcal{R}$, at least in a formal setting.
Indeed if we set $q = e^{\hbar}$ and expand up to first order in $\hbar$ we get
\begin{equation*}
\mathcal{R}^{T} \mathcal{R} = 1 \otimes 1 + \hbar(\Delta(C) - C \otimes 1 - 1 \otimes C) + O(\hbar^{2}).
\end{equation*}
Notice that this expression tells us that it is $\tau_{q}(W^{2})$ that will reduce to the Casimir operator, not $\tau_{q}(W)$. Indeed the latter will be linear in the generators and, being central, will vanish for a simple Lie algebra.
Next we look at the action of $\mathcal{R}^{T} \mathcal{R}$ on the tensor product $V(\Lambda) \otimes V(\Lambda_{0})$.
On each irreducible component $V(\mu)$ it acts as scalar multiplication by $q^{-(\Lambda, \Lambda + 2\rho) - (\Lambda_{0}, \Lambda_{0} + 2\rho) + (\mu, \mu + 2\rho)}$, see \cite[Section 8.4.3, Proposition 22]{klsc}.
In particular, upon writing $\mu = \Lambda + \lambda$ with $\lambda$ a weight of $V(\Lambda_{0})$, we can recast this scalar in the form
\begin{equation*}
r_{\lambda}
= q^{2(\Lambda, \lambda) + (\lambda, \lambda + 2\rho) - (\Lambda_{0}, \Lambda_{0} + 2\rho)},
\end{equation*}
In any case, from either expression we see that $\mathcal{R}^{T} \mathcal{R}$ is a positive diagonal operator, so that it makes sense to define $(\mathcal{R}^{T} \mathcal{R})^{t}$ for $t \in \mathbb{R}$.

With these preliminaries we can proceed to discuss our choices. First of all notice that, for what concerns the classical limit, we are free to replace the expression $(\mathcal{R}^{T} \mathcal{R} - 1) / (q - q^{-1})$ with $(\mathcal{R}^{T} \mathcal{R} - (\mathcal{R}^{T} \mathcal{R})^{-1}) / (q - q^{-1})$, since the only difference will be a factor of $2$.
This modification not essential for our results, but makes some formulae nicer and has the suggestive form of a $q$-number.
The second modification we make has important consequences, on the other hand.
As noted in the discussion on tensor products, we can define $(\mathcal{R}^{T} \mathcal{R})^{t}$ for any $t \in \mathbb{R}$.
But then it is easy to see that the construction of central elements still goes through if we use $(\mathcal{R}^{T} \mathcal{R})^{t}$ for $t \neq 1$.
Therefore this is a possibility that we might want to take into account.

\begin{remark}
It is easy to see that, in the classical limit, the invariants obtained for $t \neq 1$ differ from those obtained for $t = 1$ only by a rescaling.
On the other hand, in the quantum case different values of $t$ will give different asymptotics for these invariants, as we will see shortly.
We can appreciate this difference with the following (possibly naive) example: if the eigenvalues $\lambda_{n}$ of some classical operator are replaced by the $q$-numbers $[\lambda_{n}]$ in the quantum setting, then the transformation $\lambda_{n} \to t \lambda_{n}$ is not linear in the latter case.
In the classical Weyl's formula, such a rescaling only changes the prefactor, but not the main features of the result.
It is not difficult to imagine that this need not be the case in the quantum setting.
\end{remark}

Having discussed our choices, we can at this point simply define the element
\begin{equation*}
A_{t} = \frac{(\mathcal{R}^{T} \mathcal{R})^{t} - (\mathcal{R}^{T} \mathcal{R})^{-t}}{q - q^{-1}}.
\end{equation*}
We use a different letter from $W$ in order not to create confusion.
We can readily compute the action of $A_{t}$ on irreducible components of the tensor product $V(\Lambda) \otimes V(\Lambda_{0})$.
With conventions similar to those used above it takes the form
\begin{equation}\label{eq:a-comp}
a_{t, \lambda}
= \frac{r_{\lambda}^{t} - r_{\lambda}^{-t}}{q - q^{-1}}, \quad
r_{\lambda} = q^{2(\Lambda, \lambda) + (\lambda, \lambda + 2\rho) - (\Lambda_{0}, \Lambda_{0} + 2\rho)}.
\end{equation}
Then we define our one-parameter family of central elements by
\begin{equation}\label{eq:def-cas}
\mathcal{C}_{t} = \tau_{q}(A_{t}^{2}) =
\tau_{q} 
\left( \frac{(\mathcal{R}^{T} \mathcal{R})^{t} - (\mathcal{R}^{T} \mathcal{R})^{-t}}{q - q^{-1}} \right)^{2}.
\end{equation}
In the following we will assume, without loss of generality, that $t > 0$.

\begin{remark}
The strategy used to define $\mathcal{C}_{t}$ points to a close connection with the notion of quantum Lie algebra \cite{dego97}.
This connection seems to be reinforced by the the particular role that the adjoint representation will play, as we will see in the last section.
It would be interesting to understand how this notion fits into our picture.
\end{remark}

\subsection{Main estimate}

We now proceed to prove an estimate that will be used in the computation of the spectral dimension.
It concerns the value $\chi_{\Lambda}(\mathcal{C}_{t})$ in a representation of highest weight $\Lambda$, where $\mathcal{C}_{t}$ is defined by \eqref{eq:def-cas}. Recall that this definition depends on the choice of a fixed representation $\Lambda_{0}$. The estimate will make use of the asymptotic relation that was defined in \autoref{nota:asymp}. We also recall that we are assuming $t > 0$.

\begin{lemma}
\label{lemma:cas-ineq}
Let $V(\Lambda)$ be an $U_{q}(\mathfrak{g})$-module. Fix some arbitrary $U_{q}(\mathfrak{g})$-module $V(\Lambda_{0})$. Then we have the asymptotic relation $\chi_{\Lambda}(\mathcal{C}_{t}) \sim q^{-4t (\Lambda, \Lambda_{0})}$.
\end{lemma}

\begin{proof}
The explicit formula for $\chi_{\Lambda}(\mathcal{C}_{t})$ is obtained from equations \eqref{eq:value-c} and \eqref{eq:a-comp}.
First we show that $\dim_{q} V(\Lambda + \lambda) / \dim_{q} V(\Lambda) \sim 1$, where $\lambda$ is a weight of the fixed module $V(\Lambda_{0})$.
Using the product formula for the quantum dimension we have
\begin{equation*}
\frac{\dim_{q} V(\Lambda + \lambda)}{\dim_{q} V(\Lambda)}
= \prod_{\alpha > 0}
\frac{q^{(\Lambda + \lambda + \rho, \alpha)} - q^{-(\Lambda + \lambda + \rho, \alpha)}}{q^{(\Lambda + \rho, \alpha)} - q^{-(\Lambda + \rho, \alpha)}}.
\end{equation*}
Each term in the product can be rewritten as
\begin{equation*}
\frac{q^{2(\Lambda + \rho, \alpha)}
q^{(\lambda, \alpha)}
- q^{- (\lambda, \alpha)}}{q^{2(\Lambda + \rho, \alpha)} - 1}.
\end{equation*}
Since $0 < q < 1$, the term $q^{2(\Lambda + \rho, \alpha)}$ goes to zero for $n \to \infty$, which shows the claim.

Next we show that $a_{t, \lambda}^{2} \sim q^{-4 t (\Lambda, \Lambda_{0})}$, where $a_{t, \lambda}$ is given by \eqref{eq:a-comp}.
Consider the irreducible representation of highest weight $\Lambda_{0}$. Every weight $\lambda$ of this representation can be written in the form $\lambda = \Lambda_{0} - \sum_{i = 1}^{r} c_{i} \alpha_{i}$, where $c_{i}$ are positive integers and $\alpha_{i}$ are the simple roots, see \cite[Theorem 5.5]{knapp}. Then it immediately follows that $(\Lambda, \Lambda_{0}) \geq (\Lambda, \lambda)$ for any such $\lambda$, since the difference $\sum_{i = 1}^{r} c_{i} (\Lambda, \alpha_{i})$ is non-negative.
Now rewrite $a_{t, \alpha}^{2}$ as follows
\begin{equation*}
a_{t, \lambda}^{2} = q^{- 4 t (\Lambda, \Lambda_{0})} \left( \frac{q^{2 t (\Lambda, \Lambda_{0})} r_{\lambda}^{t}
- q^{2 t (\Lambda, \Lambda_{0})} r_{\lambda}^{-t}}{q - q^{-1}} \right)^{2}.
\end{equation*}
The product $q^{2t (\Lambda, \Lambda_{0})} r_{\lambda}^{t}$ is given explicitely by
\begin{equation*}
q^{2t (\Lambda, \Lambda_{0})} r_{\lambda}^{t}
= q^{2t (\Lambda, \Lambda_{0} + \lambda)}
q^{t (\lambda, \lambda + 2\rho) - t (\Lambda_{0}, \Lambda_{0} + 2\rho)}.
\end{equation*}
As shown above we have $(\Lambda, \Lambda_{0} + \lambda) \geq 0$ for any weight $\lambda$ of $V(\Lambda_{0})$. Therefore we have $q^{2t (\Lambda, \Lambda_{0} + \lambda)} \leq 1$, since $0 < q < 1$. This fact implies the inequalities
\begin{equation*}
0 < q^{2t (\Lambda, \Lambda_{0})} r_{\lambda}^{t} \leq
q^{t (\lambda, \lambda + 2\rho) - t (\Lambda_{0}, \Lambda_{0} + 2\rho)}.
\end{equation*}
Similar inequalities hold for the product $q^{2t (\Lambda, \Lambda_{0})} r_{\lambda}^{-t}$. From these estimates we can conclude that the big term in parentheses in the expression for $a_{t, \lambda}^{2}$ is $\sim 1$.
From this in turn we obtain that $a_{t, \lambda}^{2} \sim q^{- 4 t (\Lambda, \Lambda_{0})}$, from which the conclusion follows.
\end{proof}

\section{Results on the spectral dimension}
\label{section:spec-dim}

In this section we give some general results on the spectral dimension associated to the zeta function $\zeta_{\mathcal{C}_{t}}$.
They allow to compute this number in terms of inner products of weights of a given Lie algebra.
First we prove the result in the simple case.
Then we show that the spectral dimension of a product space is the sum of the spectral dimensions.

\subsection{Simple case}

We start by introducing some notation.

\begin{notation}
\label{not:inner-prod}
Let $G$ be a simple Lie group. Let $\Lambda = \sum_{k = 1}^{r} n_{k} \omega_{k}$ be a dominant weight and fix an an arbitrary representation of highest weight $\Lambda_{0}$.
Then we will write
\begin{equation*}
(\Lambda, 2 \rho) = \sum_{k = 1}^{r} a_{k} n_{k}, \quad
(\Lambda, \Lambda_{0}) = \sum_{k = 1}^{r} b_{k} n_{k}.
\end{equation*}
\end{notation}

Recall that $\mathcal{C}_{t}$ is the central element defined in \eqref{eq:def-cas}.

\begin{theorem}\label{thm:spec-dim}
Let $\zeta_{\mathcal{C}_{t}}$ be the zeta function corresponding to the central element $\mathcal{C}_{t}$. Then the corresponding spectral dimension is given by $p = \frac{1}{t} \max \{ \frac{a_{k}}{b_{k}} \}$.
\end{theorem}

\begin{proof}
The quantum analogue of the Peter-Weyl decomposition for $\mathbb{C}_{q}[G]$ is given by
\begin{equation*}
\mathbb{C}_{q}[G] = \bigoplus_{\Lambda \in P^{+}} V(\Lambda) \otimes V(\Lambda)^{*}.
\end{equation*}
Moreover for the quantum dimension we have $\dim_{q}(V \otimes W) = \dim_{q}V \cdot \dim_{q}W$.
From these properties it follows that we can write the trace as
\begin{equation*}
\zeta_{\mathcal{C}_{t}}(z)
= \mathrm{Tr}(\mathcal{C}_{t}^{-z/2} \modular) = \sum_{\Lambda \in P^{+}} (\dim_{q}V(\Lambda))^{2} \chi_{\Lambda}(\mathcal{C}_{t})^{-z/2}.
\end{equation*}
To find the spectral dimension we can restrict to real values $z = s \in \mathbb{R}$.

The sum over the dominant weights can be expressed as a sum over the natural numbers $\{ n_{1}, \cdots, n_{r} \}$. Here $r$ is the rank of the Lie algebra $\mathfrak{g}$. We denote these numbers schematically by $\{ n \}$.
Now we look at the asymptotics for large values of $\{ n \}$. From \autoref{lem:qdim} we have $\dim_{q} V(\Lambda) \sim q^{-(\Lambda, 2\rho)}$, while from \autoref{lemma:cas-ineq} we have $\chi_{\Lambda}(\mathcal{C}_{t}) \sim q^{-4t (\Lambda, \Lambda_{0})}$.
Then
\begin{equation*}
\zeta_{\mathcal{C}_{t}}(s)
\sim \sum_{\{ n \}}
q^{- \sum_{k = 1}^{r} 2 a_{k} n_{k}}
\left( q^{\sum_{k = 1}^{r} - 4 t b_{k} n_{k}} \right)^{-s/2}
= \sum_{\{ n \}}
q^{\sum_{k = 1}^{r} 2(s t b_{k} - a_{k}) n_{k}}.
\end{equation*}
We are left with several geometric series.
These sums are finite provided that $s > \frac{1}{t} \frac{a_{k}}{b_{k}}$ for all $k \in \{ 1, \cdots, r \}$. Then $\zeta_{\mathcal{C}_{t}}(s)$ is finite if $s > \frac{1}{t} \max \{ \frac{a_{k}}{b_{k}} \}$, from which the result follows.
\end{proof}

Notice that the spectral dimension depends on the parameter $t$. This is in contrast with the classical case, where different values of $t$ give the same spectral dimension.

\subsection{Product case}
\label{sec:product}

We will briefly consider the case of spaces which are products of simple Lie groups $G_{1} \times \cdots \times G_{n}$. It is enough to consider the case of two factors $G_{1} \times G_{2}$, with the general case being treated similarly.
This product space can be quantized in a straightforward way and we want to compute its spectral dimension.

First we have to recall how to treat products of non-commutative spaces. The algebra will be the tensor product of the corresponding quantized algebras $\mathbb{C}_{q}[G_{1}]$ and $\mathbb{C}_{q}[G_{2}]$. Similarly the Hilbert space will be the completion of this algebra with respect to the Haar states.
Then we take $\mathcal{C}_{t} = \mathcal{C}_{t, 1} \otimes 1 + 1 \otimes \mathcal{C}_{t, 2}$, which for ease of notation we write as $\mathcal{C}_{t} = \mathcal{C}_{t, 1} + \mathcal{C}_{t, 2}$. Finally for the modular operator we set $\modular = \modular^{(1)} \otimes \modular^{(2)}$.

It can be easily proven (and holds in full generality) that the dimension in this case is less or equal than the sum of the dimensions, as we show in the next lemma. 

\begin{lemma}
Let $p_{i}$ be the spectral dimension of $\zeta_{\mathcal{C}_{t, i}}(z) = \mathrm{Tr}(\mathcal{C}_{t, i}^{-z/2} \modular^{(i)})$ for $i = 1, 2$.
If we denote by $p$ the spectral dimension of $\zeta_{\mathcal{C}_{t}}(z) = \mathrm{Tr}(\mathcal{C}_{t}^{-z/2} \modular)$, then we have $p \leq p_{1} + p_{2}$.
\end{lemma}

\begin{proof}
Since $\mathcal{C}_{t, 1}$ and $\mathcal{C}_{t, 2}$ are positive operators, we have the obvious inequality
\begin{equation*}
\begin{split}
(\mathcal{C}_{t, 1} + \mathcal{C}_{t, 2})^{-1} &
= (\mathcal{C}_{t, 1} + \mathcal{C}_{t, 2})^{-p_{1}/(p_{1} + p_{2})} (\mathcal{C}_{t, 1} + \mathcal{C}_{t, 2})^{-p_{2}/(p_{1} + p_{2})} \\
& \leq \mathcal{C}_{t, 1}^{-p_{1}/(p_{1} + p_{2})} \mathcal{C}_{t, 2}^{-p_{2}/(p_{1} + p_{2})}.
\end{split}
\end{equation*}
Again it is enough to consider the case $z = s \in \mathbb{R}$.
Then we get
\begin{equation*}
\mathrm{Tr}(\mathcal{C}_{t}^{-s/2} \modular) \leq
\mathrm{Tr}(\mathcal{C}_{t, 1}^{-s/2 \ p_{1} /(p_{1} + p_{2})} \modular^{(1)})
\mathrm{Tr}(\mathcal{C}_{t, 2}^{-s/2 \ p_{2} /(p_{1} + p_{2})} \modular^{(2)}).
\end{equation*}
It follows that the right-hand side is finite for $s > p_{1} + p_{2}$. This implies $p \leq p_{1} + p_{2}$.
\end{proof}

The previous lemma does not guarantee that $p = p_{1} + p_{2}$, since in principle it might happen that $\mathcal{C}^{-z/2} \modular$ is trace-class for $z = p$, for example.
On the other hand, with a bit more work we can actually show that $p = p_{1} + p_{2}$ holds.
This is the content of the next proposition, where we give an elementary proof of this fact using the explicit form of our zeta function.

\begin{proposition}
With the same notation as above, we have $p = p_{1} + p_{2}$.
\end{proposition}

\begin{proof}
We start by outlining the strategy of the proof. First write $\zeta_{\mathcal{C}_{t}}(s) \sim A_{t}(s)$, where
\begin{equation*}
A_{t}(s)
= \sum_{ \{ n_{1} \} } \sum_{ \{ n_{2} \} }
q^{- \sum_{k = 1}^{r_{1}} 2 a_{1, k} n_{1, k}}
q^{- \sum_{k = 1}^{r_{2}} 2 a_{2, k} n_{2, k}}
\left(
q^{-\sum_{k = 1}^{r_{1}} 4 t b_{1, k} n_{1, k}}
+ q^{-\sum_{k = 1}^{r_{2}} 4 t b_{2, k} n_{2, k}}
\right)^{-s/2}.
\end{equation*}
This expression is analogue to the one given at the end of the proof of \autoref{thm:spec-dim}. To obtain it we simply use similar techniques to those of \autoref{lem:qdim} and \autoref{lemma:cas-ineq}.
Now $\zeta_{\mathcal{C}_{t}}(s) \sim A_{t}(s)$ implies in particular that $K A_{t}(s) \leq \zeta_{\mathcal{C}_{t}}(s)$ for some constant $K$.
Define $\tilde{A}_{t}(s)$ by truncating the sum over the numbers $\{ n_{2} \}$ in $A_{t}(s)$ to a finite sum.
Then clearly $K \tilde{A}_{t}(s) \leq \zeta_{\mathcal{C}_{t}}(s)$.
By choosing the truncation appropriately, we will show that $\tilde{A}_{t}(s)$ has a singularity at $s = p_{1} + p_{2}$.
Then the same will be true for $\zeta_{\mathcal{C}_{t}}(s)$, so that $p = p_{1} + p_{2}$.

We fix some further notations and conventions. We write $a_{i, k}$ and $b_{i, k}$, with $i = 1, 2$, as in \autoref{not:inner-prod}. We can assume, without loss of generality, that our labeling is such that $a_{i, 1} / b_{i, 1} \geq \cdots \geq a_{i, r_{i}} / b_{i, r_{i}}$.
Here $r_{i}$ denotes the rank of the Lie algebra $\mathfrak{g}_{i}$.
Then from \autoref{thm:spec-dim} we have that $p_{i} = \frac{1}{t} a_{i, 1} / b_{i,  1}$.
It is convenient to assume that $r_{1} \geq r_{2}$.
Finally define the numbers $c_{i, k} = b_{i, k} n_{i, k}$. The reason for this last definition will be clear shortly.
If we rewrite $A_{t}(s)$ as a sum over the numbers $c_{i, k}$ then it takes the form
\begin{equation*}
A_{t}(s)
= \sum_{ \{ c_{1} \} } \sum_{ \{ c_{2} \} }
q^{- \sum_{k = 1}^{r_{1}} 2 \frac{a_{1, k}}{b_{1, k}} c_{1, k}}
q^{- \sum_{k = 1}^{r_{2}} 2 \frac{a_{2, k}}{b_{2, k}} c_{2, k}}
\left(
q^{-\sum_{k = 1}^{r_{1}} 4 t c_{1, k}}
+ q^{-\sum_{k = 1}^{r_{2}} 4 t c_{2, k}}
\right)^{-s/2}.
\end{equation*}

Now we are ready to define the truncation $\tilde{A}_{t}(s)$. Consider the term in parentheses in the expression above for $A_{t}(s)$. We can rewrite it as follows
\begin{equation*}
q^{-\sum_{k = 1}^{r_{1}} 4 t c_{1, k}}
+ q^{-\sum_{k = 1}^{r_{2}} 4 t c_{2, k}}
= q^{-\sum_{k = 1}^{r_{1}} 4 t c_{1, k}} \left( 1 + q^{\sum_{k = 1}^{r_{2}} 4 t (c_{1, k} - c_{2, k})} q^{\sum_{k = r_{2} + 1}^{r_{1}} 4 t c_{1, k}} \right).
\end{equation*}
Here we are using that $r_{1} \geq r_{2}$. Now suppose that $c_{1, k} \geq c_{2, k}$ for each $k$. Then it is clear that the term in parentheses is bounded by two positive constants, since $0 < q < 1$, or in other words this term is $\sim 1$.
Then we define $\tilde{A}_{t}(s)$ by truncating the sum to $\{ c_{2} \leq c_{1} \}$.
The outcome of this discussion is that
\begin{equation*}
\begin{split}
\tilde{A}_{t}(s) & =
\sum_{ \{ c_{1} \} } \sum_{ \{ c_{2} \leq c_{1} \} }
q^{- \sum_{k = 1}^{r_{1}} 2 \frac{a_{1, k}}{b_{1, k}} c_{1, k}}
q^{- \sum_{k = 1}^{r_{2}} 2 \frac{a_{2, k}}{b_{2, k}} c_{2, k}} \\
& \times
q^{\sum_{k = 1}^{r_{1}} 2 s t c_{1, k}}
\left(
1 + q^{\sum_{k = 1}^{r_{2}} 4 t (c_{1, k} - c_{2, k})} q^{\sum_{k = r_{2} + 1}^{r_{1}} 4 t c_{1, k}}
\right)^{-s/2} \\
& \sim \sum_{ \{ c_{1} \} } \sum_{ \{ c_{2} \leq c_{1} \} }
q^{ \sum_{k = 1}^{r_{1}} 2 \left( s t - \frac{a_{1, k}}{b_{1, k}} \right) c_{1, k} }
q^{- \sum_{k = 1}^{r_{2}} 2 \frac{a_{2, k}}{b_{2, k}} c_{2, k}}.
\end{split}
\end{equation*}

Now we can easily compute the sums over $\{ c_{2} \}$. Indeed for each $k$ we have
\begin{equation*}
\sum_{c_{2, k} = 0}^{c_{1, k}} q^{- 2 \frac{a_{2, k}}{b_{2, k}} c_{2, k}}
\sim q^{- 2 \frac{a_{2, k}}{b_{2, k}} c_{1, k}},
\end{equation*}
which is valid for large values of $c_{1, k}$.
Using this result repeatedly we get
\begin{equation*}
\begin{split}
\tilde{A}_{t}(s)
& \sim \sum_{ \{ c_{1} \} }
q^{ \sum_{k = 1}^{r_{1}} 2 \left( s t - \frac{a_{1, k}}{b_{1, k}} \right) c_{1, k} }
q^{- \sum_{k = 1}^{r_{2}} 2 \frac{a_{2, k}}{b_{2, k}} c_{1, k}} \\
& = \sum_{ \{ c_{1} \} }
q^{ \sum_{k = 1}^{r_{2}} 2 \left( s t - \frac{a_{1, k}}{b_{1, k}} - \frac{a_{2, k}}{b_{2, k}} \right) c_{1, k} }
q^{ \sum_{k = r_{2} + 1}^{r_{1}} 2 \left( s t - \frac{a_{1, k}}{b_{1, k}} \right) c_{1, k} }.
\end{split}
\end{equation*}
These sums are convergent provided that the various exponents are positive. Now observe that, due to our labeling conventions for the coefficients $a_{i, k}$ and $b_{i, k}$, we have
\begin{equation*}
a_{1, k} / b_{1, k} + a_{2, k} / b_{2, k} \geq a_{1, k + 1} / b_{1, k + 1} + a_{2, k + 1} / b_{2, k + 1}.
\end{equation*}
These inequalities imply that all the sums are convergent provided that $s t > a_{1, 1} / b_{1, 1} + a_{2, 1} / b_{2, 1}$.
But this condition is equivalent to $s > p_{1} + p_{2}$, since we had $p_{i} = \frac{1}{t} a_{i, 1} / b_{i,  1}$.
This shows that $\tilde{A}_{t}(s)$ has a singularity at $s = p_{1} + p_{2}$, which concludes the proof.
\end{proof}

\section{Computations for quantum groups}
\label{section:comp-groups}

In this section we will perform the computations relevant for quantum groups. These will mainly be used in the following sections for the flag manifold case.

\subsection{Conventions for the roots}

In order to perform concrete computations, it is convenient to embed the weight space into some vector space $\mathcal{E}$, whose dimension depends on the Lie algebra in consideration.
Then the roots can be written in terms of an orthonormal basis of $\mathcal{E}$, which we denote by $\{ e_{i} \}$ in the table below.
Our conventions coincide with those of the book \cite[Section 1.5]{fuc} and of the Mathematica package LieART \cite{lieart}, which can be easily used to check the results we will present.
We will restrict to those simple Lie algebras that will appear in the flag manifold case.

\begin{table}[!ht]

\begin{center}
\begin{tabular}{|c|c|c|c|}
\hline 
$\mathfrak{g}$ & $\dim(\mathcal{E})$ & simple roots & positive roots \tabularnewline
\hline 
\hline

$A_{r}$ &
$r+1$ &
$e_{i} - e_{i + 1}, \quad 1 \leq i \leq r$ &
$e_{i} - e_{j}, \quad 1 \leq i < j \leq r + 1$
\tabularnewline
\hline

$B_{r}$ &
$r$ &
$e_{i} - e_{i + 1}, \quad 1 \leq i \leq r - 1,$ &
$e_{i} \pm e_{j}, \quad 1 \leq i < j \leq r,$
\tabularnewline
& &
$e_{r}$ &
$e_{i}, \quad 1 \leq i \leq r$
\tabularnewline
\hline

$C_{r}$ &
$r$ &
$e_{i} - e_{i + 1}, \quad 1 \leq i \leq r - 1,$ &
$e_{i} \pm e_{j}, \quad 1 \leq i < j \leq r,$
\tabularnewline
& &
$2e_{r}$ &
$2e_{i}, \quad 1 \leq i \leq r$
\tabularnewline
\hline

$D_{r}$ &
$r$ &
$e_{i} - e_{i + 1}, \quad 1 \leq i \leq r - 1,$ &
$e_{i} \pm e_{j}, \quad 1 \leq i < j \leq r$
\tabularnewline
& &
$e_{r - 1} + e_{r}$ &
\tabularnewline
\hline

$E_{6}$ &
$8$ &
$-e_{i} + e_{i + 1}, \quad 1 \leq i \leq 4,$ &
$e_{i} \pm e_{j}, \quad 1 \leq j < i \leq 5,$
\tabularnewline
& &
$\frac{1}{2} (e_{1} - \sum_{j = 2}^{7} e_{j} + e_{8}),$ &
$\frac{1}{2} (- \sum_{j = 1}^{5} (\pm) e_{j} - e_{6} - e_{7} + e_{8}),$
\tabularnewline
& &
$e_{1} + e_{2}$ &
\textrm{even number of minus signs}
\tabularnewline
\hline

$E_{7}$ &
$8$ &
$-e_{i} + e_{i + 1}, \quad 1 \leq i \leq 5,$ &
$e_{i} \pm e_{j}, \quad 1 \leq j < i \leq 6,$
\tabularnewline
& &
$\frac{1}{2} (e_{1} - \sum_{j = 2}^{7} e_{j} + e_{8}),$ &
$\frac{1}{2} (- \sum_{j = 1}^{6} (\pm) e_{j} - e_{7} + e_{8}),$
\tabularnewline
& &
$e_{1} + e_{2}$ &
\textrm{odd number of minus signs,}
\tabularnewline
& & &
$-e_{7} + e_{8}$
\tabularnewline
\hline
\end{tabular}
\end{center}

\caption{Expressions for the simple and positive roots.}
\end{table}

The normalization is such that the \emph{short} roots have length squared equal to two.
In the next table we give the expressions for the highest roots in these cases.
We point out that there is a $1/2$ missing in the highest root of $E_{6}$ in \cite[Section 1.5]{fuc}).

\begin{table}[!ht]

\begin{center}
\begin{tabular}{|c||c|c|c|c|c|}
\hline 
$\mathfrak{g}$ & $A_{r}$ & $B_{r}, D_{r}$ & $C_{r}$ & $E_{6}$ & $E_{7}$ \tabularnewline
\hline 
$\theta$ &
$e_{1} - e_{r + 1}$ &
$e_{1} + e_{2}$ &
$2e_{1}$ &
$\frac{1}{2} \left( \sum_{j = 1}^{5} e_{j} - e_{6} - e_{7} + e_{8} \right)$ &
$e_{8} - e_{7}$ \tabularnewline
\hline 
\end{tabular}
\end{center}

\caption{Expressions for the highest roots.}
\end{table}

Finally in \autoref{fig:dynk} we make explicit our ordering for the simple roots.
These conventions will be used throughout the paper. Of course one should keep them in mind when comparing the results presented here with other references.

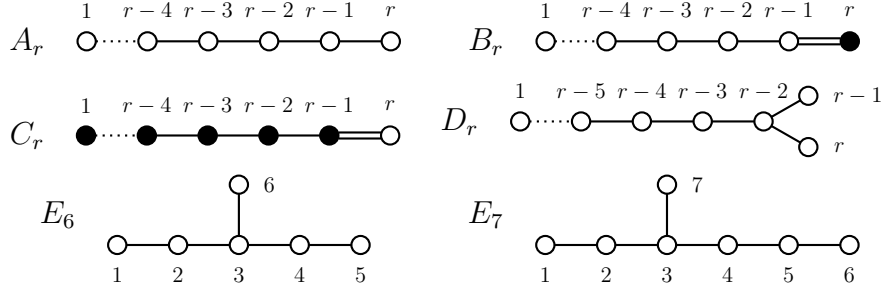
\begin{figure}[!ht]
\centering

\begin{tabular}{cc}

  \begin{tikzpicture}[scale=.4]
    \draw (-1,0) node[anchor=east]  {$A_r$};
    \foreach \x in {0,...,5}
    \draw[xshift=\x cm,thick] (\x cm,0) circle (.3cm);
    \draw[dotted,thick] (0.3 cm,0) -- +(1.4 cm,0);
    \foreach \y in {1.15,...,4.15}
    \draw[xshift=\y cm,thick] (\y cm,0) -- +(1.4 cm,0);
    
  \draw (0,1) node[scale=0.7] {$1$};
  \draw (2,1) node[scale=0.7] {$r-4$};
  \draw (4,1) node[scale=0.7] {$r-3$};
  \draw (6,1) node[scale=0.7] {$r-2$};
  \draw (8,1) node[scale=0.7] {$r-1$};
  \draw (10,1) node[scale=0.7] {$r$};
  \end{tikzpicture}

 & 
 
  \begin{tikzpicture}[scale=.4]
    \draw (-1,0) node[anchor=east]  {$B_r$};
    \foreach \x in {0,...,4}
    \draw[xshift=\x cm,thick] (\x cm,0) circle (.3cm);
    \draw[xshift=5 cm,thick,fill=black] (5 cm, 0) circle (.3 cm);
    \draw[dotted,thick] (0.3 cm,0) -- +(1.4 cm,0);
    \foreach \y in {1.15,...,3.15}
    \draw[xshift=\y cm,thick] (\y cm,0) -- +(1.4 cm,0);
    \draw[thick] (8.3 cm, .1 cm) -- +(1.4 cm,0);
    \draw[thick] (8.3 cm, -.1 cm) -- +(1.4 cm,0);
    
  \draw (0,1) node[scale=0.7] {$1$};
  \draw (2,1) node[scale=0.7] {$r-4$};
  \draw (4,1) node[scale=0.7] {$r-3$};
  \draw (6,1) node[scale=0.7] {$r-2$};
  \draw (8,1) node[scale=0.7] {$r-1$};
  \draw (10,1) node[scale=0.7] {$r$};
  \end{tikzpicture}

 \tabularnewline
 
  \begin{tikzpicture}[scale=.4]
    \draw (-1,0) node[anchor=east]  {$C_r$};
    \foreach \x in {0,...,4}
    \draw[xshift=\x cm,thick,fill=black] (\x cm,0) circle (.3cm);
    \draw[xshift=5 cm,thick] (5 cm, 0) circle (.3 cm);
    \draw[dotted,thick] (0.3 cm,0) -- +(1.4 cm,0);
    \foreach \y in {1.15,...,3.15}
    \draw[xshift=\y cm,thick] (\y cm,0) -- +(1.4 cm,0);
    \draw[thick] (8.3 cm, .1 cm) -- +(1.4 cm,0);
    \draw[thick] (8.3 cm, -.1 cm) -- +(1.4 cm,0);
    
  \draw (0,1) node[scale=0.7] {$1$};
  \draw (2,1) node[scale=0.7] {$r-4$};
  \draw (4,1) node[scale=0.7] {$r-3$};
  \draw (6,1) node[scale=0.7] {$r-2$};
  \draw (8,1) node[scale=0.7] {$r-1$};
  \draw (10,1) node[scale=0.7] {$r$};
  \end{tikzpicture}
 
 & 
 
  \begin{tikzpicture}[scale=.4]
    \draw (-1,0) node[anchor=east]  {$D_r$};
    \foreach \x in {0,...,4}
    \draw[xshift=\x cm,thick] (\x cm,0) circle (.3cm);
    \draw[xshift=8 cm,thick] (30: 17 mm) circle (.3cm);
    \draw[xshift=8 cm,thick] (-30: 17 mm) circle (.3cm);
    \draw[dotted,thick] (0.3 cm,0) -- +(1.4 cm,0);
    \foreach \y in {1.15,...,3.15}
    \draw[xshift=\y cm,thick] (\y cm,0) -- +(1.4 cm,0);
    \draw[xshift=8 cm,thick] (30: 3 mm) -- (30: 14 mm);
    \draw[xshift=8 cm,thick] (-30: 3 mm) -- (-30: 14 mm);
    
  \draw (0,1) node[scale=0.7] {$1$};
  \draw (2,1) node[scale=0.7] {$r-5$};
  \draw (4,1) node[scale=0.7] {$r-4$};
  \draw (6,1) node[scale=0.7] {$r-3$};
  \draw (8,1) node[scale=0.7] {$r-2$};
  \draw (11.1,0.85) node[scale=0.7] {$r-1$};
  \draw (10.5,-0.85) node[scale=0.7] {$r$};
  \end{tikzpicture}
 
 \tabularnewline
 
  \begin{tikzpicture}[scale=.4]
    \draw (-1,1) node[anchor=east]  {$E_6$};
    \foreach \x in {0,...,4}
    \draw[thick,xshift=\x cm] (\x cm,0) circle (3 mm);
    \foreach \y in {0,...,3}
    \draw[thick,xshift=\y cm] (\y cm,0) ++(.3 cm, 0) -- +(14 mm,0);
    \draw[thick] (4 cm,2 cm) circle (3 mm);
    \draw[thick] (4 cm, 3mm) -- +(0, 1.4 cm);
    
  \draw (0,-1) node[scale=0.7] {$1$};
  \draw (2,-1) node[scale=0.7] {$2$};
  \draw (4,-1) node[scale=0.7] {$3$};
  \draw (6,-1) node[scale=0.7] {$4$};
  \draw (8,-1) node[scale=0.7] {$5$};
  \draw (5,2) node[scale=0.7] {$6$};
  \end{tikzpicture}
 
 & 

  \begin{tikzpicture}[scale=.4]
    \draw (-1,1) node[anchor=east]  {$E_7$};
    \foreach \x in {0,...,5}
    \draw[thick,xshift=\x cm] (\x cm,0) circle (3 mm);
    \foreach \y in {0,...,4}
    \draw[thick,xshift=\y cm] (\y cm,0) ++(.3 cm, 0) -- +(14 mm,0);
    \draw[thick] (4 cm,2 cm) circle (3 mm);
    \draw[thick] (4 cm, 3mm) -- +(0, 1.4 cm);
    
  \draw (0,-1) node[scale=0.7] {$1$};
  \draw (2,-1) node[scale=0.7] {$2$};
  \draw (4,-1) node[scale=0.7] {$3$};
  \draw (6,-1) node[scale=0.7] {$4$};
  \draw (8,-1) node[scale=0.7] {$5$};
  \draw (10,-1) node[scale=0.7] {$6$};
  \draw (5,2) node[scale=0.7] {$7$};
  \end{tikzpicture}
 
 \tabularnewline
\end{tabular}

\caption{Dynkin diagrams for the simple Lie algebras of interest.}
\label{fig:dynk}
\end{figure}

\subsection{Inner products}

Now that we have fixed the conventions, we can start computing the relevant inner products.
Recall that from \autoref{thm:spec-dim} we know that these are $(\Lambda, 2 \rho)$, coming from the quantum dimension, and $(\Lambda, \Lambda_{0})$, coming from the operator $\mathcal{C}_{t}$.
Here we will simply state the results, since the computations are not very illuminating.
Once again, we mention that they can be checked with the Mathematica package LieART \cite{lieart}.

We start with the inner product coming from the quantum dimension. Recall that we write $\Lambda = \sum_{k = 1}^{r} n_{k} \omega_{k}$ for a general dominant weight.

\begin{lemma}
\label{lemma:scal-g}
Let $\Lambda$ be a dominant weight. Then $(\Lambda, 2\rho)$ is given by

\begin{equation*}
\begin{array}{ll}
\sum_{k=1}^{r} k (r + 1 - k) n_{k} & \mathfrak{g} = A_{r}, \\

\sum_{k=1}^{r-1} k (2r - k) n_{k} + \frac{1}{2} r^{2} n_{r} & \mathfrak{g} = B_{r}, \\

\sum_{k=1}^{r} k (2r + 1 - k) n_{k} & \mathfrak{g} = C_{r}, \\

\sum_{k=1}^{r-2} k (2r - 1 - k) n_{k} + \frac{1}{2} r(r-1) (n_{r-1} + n_{r}) & \mathfrak{g} = D_{r}, \\

16 n_{1} + 30 n_{2} + 42 n_{3} + 30 n_{4} + 16 n_{5} + 22 n_{6} & \mathfrak{g} = E_{6}, \\
 
34 n_{1} + 66 n_{2} + 96 n_{3} + 75 n_{4} + 52 n_{5} + 27 n_{6} + 49 n_{7} & \mathfrak{g} = E_{7}.
\end{array}
\end{equation*}
\end{lemma}

\begin{proof}
By direct computation.
\end{proof}

Now, before computing the inner product $(\Lambda, \Lambda_{0})$, clearly we have to make a choice for the fixed representation $\Lambda_{0}$. Here we will present the results in the case of the fundamental and adjoint representations, which are the two natural ones to consider.

We start with the fundamental representation. We denote the highest weight of this representation by $\Lambda_{F}$. In terms of the Dynkin labels it is given by $(1, 0, \cdots, 0)$ for the simple Lie algebras $A_{r}$ (for $r \geq 1$), $B_{r}$ (for $r \geq 2$), $C_{r}$ (for $r \geq 1$) and $D_{r}$ (for $r \geq 3$). For the exceptional Lie algebras $E_{6}$ and $E_{7}$ we use respectively $(1, 0, 0, 0, 0, 0)$ and $(0, 0, 0, 0, 0, 1, 0)$.
We should mention that for $E_{6}$ there is another "fundamental" representation of dimension $27$, but this distinction is not important for our purposes.

\begin{lemma}
\label{lem:inn-fund}
Let $\Lambda$ be a dominant weight. Then $(\Lambda, \Lambda_{F})$ is given by
\begin{equation*}
\begin{array}{ll}
\sum_{k = 1}^{r} \frac{r + 1 - k}{r + 1} n_{k}, &
\mathfrak{g} = A_{r}, \\

\sum_{k = 1}^{r - 1} n_{k} + \frac{1}{2} n_{r}, &
\mathfrak{g} = B_{r}, \\

\sum_{k=1}^{r} n_{k}, &
\mathfrak{g} = C_{r}, \\

\sum_{k = 1}^{r - 2} n_{k} + \frac{1}{2}(n_{r - 1} + n_{r}), &
\mathfrak{g} = D_{r}, \\

\frac{4}{3} n_{1} + \frac{5}{3} n_{2} + 2 n_{3} + \frac{4}{3} n_{4} + \frac{2}{3} n_{5} + n_{6}, & \mathfrak{g} = E_{6}, \\
 
n_{1} + 2 n_{2} + 3 n_{3} + \frac{5}{2} n_{4} + 2 n_{5} + \frac{3}{2} n_{6} + \frac{3}{2} n_{7}, & \mathfrak{g} = E_{7}.
\end{array}
\end{equation*}

\end{lemma}

\begin{proof}
By direct computation.
\end{proof}

Similarly, we repeat this computation for the adjoint representation. In this case the highest weight coincides with the highest root, which we denote by $\theta$.

\begin{lemma}
\label{lem:inn-adj}
Let $\Lambda$ be a dominant weight. Then $(\Lambda, \theta)$ is given by
\begin{equation*}
\begin{array}{ll}
\sum_{k=1}^{r} n_{k}, &
\mathfrak{g} = A_{r}, \\

n_{1} + \sum_{k=2}^{r-1} 2 n_{k} + n_{r}, &
\mathfrak{g} = B_{r}, \\

\sum_{k=1}^{r} 2 n_{k}, &
\mathfrak{g} = C_{r}, \\

n_{1} + \sum_{k = 2}^{r - 2} 2 n_{k} + n_{r - 1} + n_{r}, &
\mathfrak{g} = D_{r}, \\

n_{1} + 2 n_{2} + 3 n_{3} + 2 n_{4} + n_{5} + 2 n_{6}, & \mathfrak{g} = E_{6}, \\
 
2 n_{1} + 3 n_{2} + 4 n_{3} + 3 n_{4} + 2 n_{5} + n_{6} + 2 n_{7}, & \mathfrak{g} = E_{7}.
\end{array}
\end{equation*}

\end{lemma}

\begin{proof}
By direct computation.
\end{proof}

At this point we could give the results for the spectral dimension in these two cases. They can be obtained simply by plugging the numbers given above into \autoref{thm:spec-dim}.
However we will not do this. The point is that in this case there is no clear relation between the spectral dimension, obtained by this procedure, and the classical dimension.
This is not unexpected, in view of similar results known in the literature.
Here we will use these computations as an intermediate step for the flag manifold case.
In this case, in view of the results mentioned in the introduction, we have better chances of getting something interesting.

In closing this section, we would like to point out the reference \cite{mat13}, where a similar problem is considered for the quantum group $\mathbb{C}_{q}[SU(2)]$. But there are important differences: first an operator different from $\modular$ appears in the definition of the zeta functions; secondly the analogue of the operator $\mathcal{C}$ is not central. Nevertheless, in this case one can still recover the Haar state, and moreover the spectral dimension coincides with the classical one. It is not clear whether this approach can be generalized to other quantum groups.
We want to point out, however, a certain compatibility with the present paper: indeed, upon restriction to the Podle\'{s} sphere, we get exactly the setting that we are considering here.

\section{Irreducible flag manifolds}
\label{section:flag-man}

\subsection{Definition and classification}

We start by recalling the relevant definitions.
Let $\mathfrak{g}$ be a complex semisimple Lie algebra. Denote by $G_{\mathbb{C}}$ the corresponding simply connected Lie group.
Let $\mathfrak{p}$ be a parabolic subalgebra, that is any subalgebra that contains a Borel subalgebra. Denote by $P$ the corresponding subgroup of $G_{\mathbb{C}}$. Then a \emph{generalized flag manifold} is defined to be the complex manifold $G_{\mathbb{C}} / P$.
As a real manifold it is diffeomorphic to $G / K$, where $G$ is the compact real form of $G_{\mathbb{C}}$ and $K = L \cap G$ is the real form of the Levi factor $L$.
The case when the adjoint action of $\mathfrak{p}$ on $\mathfrak{g} / \mathfrak{p}$ is irreducible (in which case we say that $\mathfrak{p}$ is of cominuscole type) corresponds to $G_{\mathbb{C}} / P$ being a \emph{symmetric space}. If in addition $\mathfrak{g}$ is simple, then the generalized flag manifold is called \emph{irreducible}.

The class of irreducible generalized flag manifolds coincides with the class of compact irreducible Hermitian symmetric spaces.
The latter are Riemannian symmetric spaces admitting a complex structure, which furthermore can not be written as a product of spaces of the same type.
For a comprehensive reference on these topics we refer to the book \cite{hel}.
The complete list of these irreducible components can be obtained from the classification of Riemannian symmetric spaces given by Cartan. It is provided in the table below.

\begin{table}[!ht]

\begin{center}
\begin{tabular}{|c|c|c|c|c|}
\hline 
name & $G$ & $K$ & for & real dimension \tabularnewline
\hline 
\hline

AIII & $SU(p+q)$ & $S(U(p) \times U(q))$ & $p > q \geq 1$ & $2 pq$ \tabularnewline
\hline

BDI ($q = 2$) & $SO(p + 2)$ & $SO(p) \times SO(2)$ & $p \geq 3$ & $2p$ \tabularnewline
\hline

CI & $Sp(r$) & $U(r)$ & $r \geq 2$ & $r(r + 1)$ \tabularnewline
\hline

DIII & $SO(2r)$ & $U(r)$ & $r \geq 3$ & $r(r - 1)$ \tabularnewline
\hline

EIII & $E_{6}$ & $SO(10) \times SO(2)$ & & $32$ \tabularnewline
\hline

EVII & $E_{7}$ & $E_{6} \times SO(2)$ & & $54$ \tabularnewline
\hline
\end{tabular}
\end{center}

\caption{List of compact irreducible Hermitian symmetric spaces.}
\end{table}

The restrictions on the parameters are to avoid some of the low-dimensional isomorphisms, see \cite[Chapter X.6.4]{hel}, and will be enforced for the rest of the paper.
We also introduce some further notation.
For $AIII$ we set $p+q = r+1$, while for $BDI$ we set $r = p/2 + 1$ is $p$ is even and $r = (p + 1)/2$ is $p$ is odd. With this notation the number $r$ coincides with the rank of the corresponding Lie algebras in all cases.

\subsection{Spherical weights}

The Peter-Weyl theorem for $\mathbb{C}[G]$ allows to derive a similar decomposition for $\mathbb{C}[G/K]$. A dominant weight $\lambda \in P^{+}$ is called \emph{spherical} if the subspace of $K$-fixed vectors in $V(\lambda)$ is one-dimensional.
Then the corresponding representation $V(\lambda)$ is also called spherical. We denote by $P_{K}^{+} \subset P^{+}$ the subset of dominant spherical weights. Then we have
\begin{equation*}
\mathbb{C}[G/K] \simeq \bigoplus_{\lambda \in P_{K}^{+}} V(\lambda).
\end{equation*}
Indeed if $G/K$ is an irreducible compact Hermitian symmetric space then $(G,K)$ is a Gelfand pair.
It is possible to introduce a subset $\{ \sph_{1}, \cdots, \sph_{n} \}$ of $P_{K}^{+}$ such that each $\lambda \in P_{K}^{+}$ can be written as a linear combination of these weights with non-negative coefficients. We call $\{ \sph_{i} \}$ the \emph{fundamental spherical weights}.
Clearly they can be expressed in terms the fundamental weights. Explicit formulae are given in the next lemma.

\begin{lemma}
The fundamental spherical weights $\{ \sph_{i} \}$ are given by
\begin{equation*}
\begin{array}{ll}
\sph_{1} = \omega_{1} + \omega_{r},\ 
\sph_{2} = \omega_{2} + \omega_{r-1},\ 
\cdots,\ 
\sph_{q} = \omega_{q} + \omega_{p},  
& AIII, \\

\sph_{1} = 2 \omega_{1},\ 
\sph_{2} = \omega_{2},
& BDI (q=2), \\

\sph_{1} = 2 \omega_{1},\ 
\sph_{2} = 2 \omega_{2},\ 
\cdots,\
\sph_{r} = 2 \omega_{r},
& CI, \\

\sph_{1} = \omega_{2},\ 
\sph_{2} = \omega_{4},\ 
\cdots,\ 
\sph_{\ell - 1} = \omega_{r - 2},\ 
\sph_{\ell} = 2 \omega_{r},
& DIII,\ r = 2\ell, \\

\sph_{1} = \omega_{2},\ 
\sph_{2} = \omega_{4},\ 
\cdots,\ 
\sph_{\ell - 1} = \omega_{r - 3},\ 
\sph_{\ell} = \omega_{r - 1} + \omega_{r},
& DIII,\ r = 2\ell + 1, \\

\sph_{1} = \omega_{1} + \omega_{5},\ 
\sph_{2} = \omega_{6},
& EIII, \\
 
\sph_{1} = \omega_{1},\ 
\sph_{2} = \omega_{5},\ 
\sph_{3} = 2 \omega_{6},
& EVII.
\end{array}
\end{equation*}
\end{lemma}

\begin{proof}
A list of fundamental spherical weights is given in \cite[Tabelle 1]{kra79}, although our conventions are slightly different.
Alternatively one can obtain the result using Satake diagrams, which classify reals forms of complex simple Lie algebras.
For details and their complete list see for example \cite[Chapter X.F]{hel}.
Then an algorithm for obtaining the fundamental spherical weights from the Satake diagrams is given in \cite[Theorem 3]{sug62}.
\end{proof}

\begin{notation}
We will write a general dominant spherical weight as $\Lambda_{S} = \sum_{k = 1}^{n} m_{k} \sph_{k}$, where $\{ \sph_{k} \}$ are the fundamental spherical weights and $\{ m_{k} \}$ are non-negative integers.
\end{notation}

In the following we will need to compute some inner products involving spherical weights.
These can be obtained from the inner products for the corresponding simple Lie algebra $\mathfrak{g}$ as follows.
Let us write $\sph_{k} = \sum_{j = 1}^{r} c_{kj} \omega_{j}$ in terms of the fundamental weights. Then we have
\begin{equation*}
(\Lambda_{S}, \gamma) = \sum_{k = 1}^{n} m_{k} (\sph_{k}, \gamma)
= \sum_{k = 1}^{n} \sum_{j = 1}^{r} c_{kj} m_{k} (\omega_{j}, \gamma),
\end{equation*}
where $\gamma$ is any weight.
Therefore, if we write $(\Lambda, \gamma) = \sum_{j = 1}^{r} n_{j} (\omega_{j}, \gamma)$, then the inner product $(\Lambda_{S}, \gamma)$ is obtained from $(\Lambda, \gamma)$ by replacing $n_{j}$ with $\sum_{k = 1}^{n} c_{kj} m_{k}$.

\subsection{Computation quantum dimension}

The structure of irreducible generalized flag manifolds presented above carries over to the quantum setting, see \cite[Proposition 4.1]{stdi99}. Therefore we have the decomposition in terms of spherical weights
\begin{equation*}
\mathbb{C}_{q}[G/K] \simeq \bigoplus_{\lambda \in P_{K}^{+}} V(\lambda).
\end{equation*}
This is all we need in order to compute the quantum dimension.

\begin{lemma}
\label{lem:scal-sph}
Let $\Lambda_{S}$ be a dominant spherical weight. Then $(\Lambda_{S}, 2\rho)$ is given by
\begin{equation*}
\begin{array}{ll}
\sum_{k=1}^{q} 2k (r + 1 - k) m_{k},
& AIII, \\

2p m_{1} + 2 (p - 1) m_{2},
& BDI(q = 2), \\

\sum_{k = 1}^{r} 2k (2r + 1 - k) m_{k},
& CI, \\

\sum_{k = 1}^{\ell} 2k (4\ell - 1 - 2k) m_{k},
& DIII,\ r = 2\ell, \\

\sum_{k = 1}^{l} 2k (4\ell + 1 - 2k) m_{k},
& DIII,\ r = 2\ell + 1, \\

32 m_{1} + 22 m_{2},
& EIII, \\
 
34 m_{1} + 52 m_{2} + 54 m_{3},
& EVII.
\end{array}
\end{equation*}
\end{lemma}

\begin{proof}
The result follows by applying the substitution rule described above to \autoref{lemma:scal-g}.
Note that in the case $AIII$ the factor $2$ arises from the invariance of the coefficient $k (r + 1 - k)$ under $k \to r + 1 - k$.
We also remark that the case $BDI(q = 2)$ needs to be treated with some care.
This is because the corresponding Lie algebra is $\mathfrak{so}(p + 2)$ so that, depending on the parity of $p$, we have to deal with two different series.
In the case $p = 2\ell$ we get $D_{r} = \mathfrak{so}(2\ell + 2) $ with $r = p/2 + 1$, while in the case $p = 2\ell + 1$ we get $B_{r} = \mathfrak{so}(2l + 3)$ with $r = (p + 1)/2$.
The result turns out to take the same form for in both cases.
\end{proof}

From the above lemma we easily obtain the following interesting result.

\begin{proposition}
\label{proposition:qdim-asym}
Let $\Lambda_{S}$ be a dominant spherical weight. Write $(\Lambda_{S}, 2\rho) = \sum_{k} a_{k}^{S} m_{k}$ for some non-negative integers $\{ a^{S}_{k} \}$ as in the previous lemma.
Then we have $\frac{2}{(\theta, \theta)} \max\{ a_{k}^{S} \} = d$, where $d$ is the dimension of the corresponding classical space.
\end{proposition}

\begin{proof}
From \autoref{lem:scal-sph} we easily obtain the following values for $\max\{ a_{k}^{S} \}$: $2pq$ for $AIII$, $2n$ for $BDI$, $2r(r+1)$ for $CI$, $r (r - 1)$ for $DIII$, $32$ for $EIII$ and $54$ for $EVII$.
For the highest root we have $(\theta, \theta) = 2$ for all simple algebras except for $C_{r}$, for which $(\theta, \theta) = 4$.
The conclusion follows by comparison with the list of irreducible Hermitian symmetric spaces.
\end{proof}

\begin{remark}
It is important to notice that the equality $\frac{2}{(\theta, \theta)} \max\{ a_{k}^{S} \} = d$ does not depend on the choice of the inner product $(\cdot, \cdot)$, since it is given by a ratio.
\end{remark}

\section{Computations for quantized flag manifolds}
\label{section:comp-flag}

In this section we will conclude the computation of the spectral dimension for quantized irreducible generalized flag manifolds. Again, we will consider the case of the fundamental and adjoint representations.
We begin with some notation.

\begin{notation}
Let $G/K$ be an irreducible generalized flag manifold. Then for any dominant spherical weight $\Lambda_{S} = \sum_{k = 1}^{n} m_{k} \sph_{k}$ and any fixed representation $\Lambda_{0}$ we will write
\begin{equation*}
(\Lambda_{S}, 2 \rho) = \sum_{k = 1}^{n} a_{k}^{S} m_{k}, \quad
(\Lambda_{S}, \Lambda_{0}) = \sum_{k = 1}^{n} b_{k}^{S} m_{k}.
\end{equation*}
\end{notation}

We can easily adapt the formula for the spectral dimension, given in \autoref{thm:spec-dim}, to the case of quantized flag manifolds. In this case, the trace appearing in the definition $\zeta_{\mathcal{C}_{t}}(z) = \mathrm{Tr}(\mathcal{C}_{t}^{- z/2} \modular)$ must be taken over the Hilbert space completion of $\mathbb{C}_{q}[G/K]$.

\begin{corollary}
\label{cor:spec-flag}
Let $G/K$ be an irreducible generalized flag manifold.
Then the spectral dimension corresponding to the zeta function $\zeta_{\mathcal{C}_{t}}$ is given by $p = \frac{1}{2 t} \max \{ a_{k}^{S} / b_{k}^{S} \}$.
\end{corollary}

\begin{proof}
Recall that $\mathbb{C}_{q}[G/K]$ has a multiplicity-free decomposition given by
\begin{equation*}
\mathbb{C}_{q}[G/K] = \bigoplus_{\Lambda_{S} \in P^{+}_{K}} V(\Lambda_{S}),
\end{equation*}
where the sum is over all dominant spherical weights. Then the only difference with the quantum group case, given in \autoref{thm:spec-dim}, is that we get $\dim_{q} V(\Lambda_{S})$ instead of $(\dim_{q}V(\Lambda))^{2}$. Therefore this factor of $2$ must be accounted for in the rest of the proof. Once this is done, by replacing the coefficients with their spherical counterparts we obtain the result.
\end{proof}

Next we investigate the fundamental and adjoint representations.

\subsection{Fundamental representation}

We start with the fundamental representation and use the notation $\Lambda_{F}$ for its highest weight. This, as it will turn out, is not the most interesting case, but it allows comparisons with other results in the literature.

\begin{lemma}
Let $\Lambda_{S}$ be a dominant spherical weight. 
Then $(\Lambda_{S}, \Lambda_{F})$ is given by
\begin{equation*}
\begin{array}{ll}
\sum_{k = 1}^{q} m_{k},
& AIII, \\

2 m_{1} + m_{2},
& BDI(q = 2), \\

\sum_{k = 1}^{r} 2 m_{k},
& CI, \\

\sum_{k = 1}^{\ell} m_{k},
& DIII, \ r = 2\ell \\

\sum_{k = 1}^{\ell} m_{k},
& DIII, \ r = 2\ell + 1 \\

2 m_{1} + m_{2},
& EIII, \\
 
m_{1} + 2 m_{2} + 3 m_{3},
& EVII.
\end{array}
\end{equation*}
In particular for the cases $AIII$, $CI$ and $DIII$ we have $(\Lambda_{S}, \Lambda_{F}) = \frac{1}{2} (\theta, \theta) \sum_{k = 1}^{n} m_{k}$.
\end{lemma}

\begin{proof}
It follows by plugging the spherical weights into the formulae given in \autoref{lem:inn-fund}.
The fact that $(\Lambda_{S}, \Lambda_{F}) = \frac{1}{2} (\theta, \theta) \sum_{k = 1}^{n} m_{k}$ follows by observing that, in our conventions, we have $(\theta, \theta) = 2$ for all simple Lie algebras except for $C_{r}$, for which $(\theta, \theta) = 4$.
\end{proof}

By plugging the results of \autoref{proposition:qdim-asym} and the above lemma into \autoref{cor:spec-flag} we obtain, for the cases $AIII$, $CI$ and $DIII$, the following expression for the spectral dimension
\begin{equation*}
p = \frac{1}{2 t} \max \{ a_{k}^{S} / b_{k}^{S} \} = \frac{1}{t (\theta, \theta)} \max \{ a_{k}^{S} \} = \frac{d}{2 t}.
\end{equation*}
As we will see in a moment, a similar pattern appears for the adjoint representation. Notice that for the other cases the formula is more complicated, since $\max \{ a_{k}^{S} / b_{k}^{S} \} \neq \frac{2}{(\theta, \theta)} \max \{ a_{k}^{S} \}$.

\begin{remark}
A computation of the spectral dimension for quantum projective spaces was given in \cite{mat14}. It can be checked that the Casimir operator used in this reference corresponds, in the language of the present paper, to the choice of the fundamental representation and $t = 1/2$.
Indeed in this case the spectral dimension coincides with the classical one.
\end{remark}

\subsection{Adjoint representation}

Next we turn to the adjoint representation. The computations are completely parallel to those for the fundamental one.

\begin{lemma}
\label{lemma:cas-asym}
Let $\Lambda_{S}$ be a dominant spherical weight. 
Then $(\Lambda_{S}, \theta)$ is given by
\begin{equation*}
\begin{array}{ll}
\sum_{k = 1}^{q} 2 m_{k},
& AIII, \\

2 m_{1} + 2 m_{2},
& BDI(q = 2), \\

\sum_{k = 1}^{r} 4 m_{k},
& CI, \\

\sum_{k = 1}^{\ell} 2 m_{k},
& DIII, \ r = 2\ell \\

\sum_{k = 1}^{\ell} 2 m_{k},
& DIII, \ r = 2\ell + 1 \\

2 m_{1} + 2 m_{2},
& EIII, \\
 
2 m_{1} + 2 m_{2} + 2 m_{3},
& EVII.
\end{array}
\end{equation*}
In particular we have $(\Lambda_{S}, \theta) = (\theta, \theta) \sum_{k = 1}^{n} m_{k}$.
\end{lemma}

\begin{proof}
It follows by plugging the spherical weights into the formulae given in \autoref{lem:inn-adj}.
The fact that $(\Lambda_{S}, \theta) = (\theta, \theta) \sum_{k = 1}^{n} m_{k}$ follows by observing that, in our conventions, we have $(\theta, \theta) = 2$ for all simple Lie algebras except for $C_{r}$, for which $(\theta, \theta) = 4$.
\end{proof}

Notice that in this case, in contrast with the fundamental representation, the result takes the same form for all spaces. Then similarly, by plugging the results of \autoref{proposition:qdim-asym} and the above lemma into \autoref{cor:spec-flag}, we obtain the spectral dimension
\begin{equation*}
p = \frac{1}{2 t} \max \{ a_{k}^{S} / b_{k}^{S} \} = \frac{1}{2 t (\theta, \theta)} \max \{ a_{k}^{S} \} = \frac{d}{4 t}.
\end{equation*}
In particular if we set $t = 1/4$ the spectral dimension coincides with the classical one.
We summarize this and the previous results into the following theorem, which is our analogue of Weyl's law for quantized irreducible generalized flag manifolds.

\begin{theorem}
\label{theorem:weyl-law}
Let $G/K$ be an irreducible generalized flag manifold of dimension $d$.
Let $\mathcal{C}$ be the central element defined in \eqref{eq:def-cas} for the value $t = 1/4$, acting on the Hilbert space completion of $\mathbb{C}_{q}[G/K]$.
Define the zeta function $\zeta_{\mathcal{C}, a}(z) = \mathrm{Tr}(a \mathcal{C}^{-z/2} \modular)$, where $a \in \mathbb{C}_{q}[G/K]$.
It satisfies the following properties, in analogy with the residue formula \eqref{eq:res-form}:
\begin{itemize}
\item we have the equality $\zeta_{\mathcal{C}, a}(z) = \zeta_{\mathcal{C}}(z) h(a)$, where $h$ denotes the Haar state,
\item $\zeta_{\mathcal{C}}(z)$ is holomorphic for $\mathrm{Re}(z) > d$ and has a singularity at $z = d$.
\end{itemize}
\end{theorem}

While we have stated this result for the value $t = 1/4$, other choices might be of interest. For example it is possible to show that for $t = 1/2$ the eigenvalues of $\mathcal{C}_{t}$ take a particularly nice form. In this case the spectral dimension corresponds to the \emph{complex} dimension of $G/K$.
In any case the general result is that, for any value of $t$, the spectral dimension is proportional to the classical one.
It is important to stress that this is valid for \emph{all} quantized irreducible generalized flag manifolds, since of course we can always choose a value of $t$ for one particular space in such a way that the classical dimension is obtained.

The general setting for zeta functions developed in this paper can be explored further. For example one can investigate the nature of the singularity at the spectral dimension and, more generally, the possibility of a meromorphic extension of these functions. In this case the residues at the poles might contain interesting information, as one would expect from the classical setting. Some of these matters will be investigated in a future publication.

\vspace{3mm}

{\footnotesize
\emph{Acknowledgements}. I wish to thank Sergey Neshveyev for a discussion regarding \autoref{prop:res-int}. I am supported by the European Research Council under the European Union's Seventh Framework Programme (FP/2007-2013) / ERC Grant Agreement no. 307663 (P.I.: S. Neshveyev).
I also wish to thank the Hausdorff Research Institute for Mathematics in Bonn and the organizers of the 2014 Trimester Program "Non-commutative Geometry and its Applications", where part of this work was done.

}

\bigskip

\end{document}